\documentclass[a4paper,12pt]{article}
\usepackage[utf8]{inputenc}
\usepackage[english]{babel}
\usepackage{amsmath, amsfonts, amsthm, amssymb}
\usepackage[labelfont=bf]{caption,subcaption}
\usepackage{color,soul}
\usepackage{enumerate}
\usepackage{float}
\usepackage[margin=2.5cm]{geometry}
\usepackage{graphicx}
\usepackage{hyperref}
\usepackage[nameinlink]{cleveref}
\usepackage{mathrsfs}
\usepackage{pdfpages}
\usepackage{tikz-cd}

	\title{The density and minimal gap of visible points in some planar quasicrystals \thanks{Partially supported by the Swedish Research Council Grant 2016-03360.}}
	\author{Gustav Hammarhjelm}
		
	\newcommand{\ASL}[2]{\mathrm{ASL}(#1,#2)}
	
	\newcommand{\bb}[1]{\mathbb{#1}}
	\newcommand{\bs}{\backslash}
	\newcommand{\ceil}[1]{\lceil#1\rceil}

	\newcommand{\GL}[2]{\mathrm{GL}(#1,#2)}

	\newcommand{\inv}[1]{#1^{-1}}

	\newcommand{\liminfl}[1]{\underset{#1\to\infty}{\liminf}\,}
	\newcommand{\limsupl}[1]{\underset{#1\to\infty}{\limsup}\,}
	\newcommand{\mc}[1]{\mathcal{#1}}
	\newcommand{\norm}[1]{\left\| #1\right\|}

	\newcommand{\SL}[2]{\mathrm{SL}(#1,#2)}
	
	\newcommand{\smpt}[1]{\setminus\{#1\}}

	\newcommand{\vol}{\mathrm{vol}}

	\renewcommand{\a}{\alpha}
	
	\renewcommand{\d}{\delta}
	\newcommand{\e}{\epsilon}
	\newcommand{\g}{\gamma}
	\renewcommand{\k}{\kappa}

	\newcommand{\s}{\sigma}
	\renewcommand{\t}{\tau}
	\newcommand{\z}{\zeta}
	\newcommand{\C}{\bb{C}}
	\newcommand{\Q}{\bb{Q}}
	\newcommand{\R}{\bb{R}}
	
	\newcommand{\Z}{\bb{Z}}

	\newtheorem{thm}{Theorem}[section]
	\newtheorem*{thm*}{Theorem}

	\newtheorem{lem}[thm]{Lemma}
	\newtheorem{cor}[thm]{Corollary}
	\newtheorem{prop}[thm]{Proposition}
	
	\theoremstyle{definition}
	\newtheorem{defn}[thm]{Definition}
	
	\theoremstyle{remark}
	\newtheorem*{rem}{Remark}
	
	\theoremstyle{remark}
	
	\theoremstyle{remark}
	
	\renewcommand\footnotemark{}
	
\begin{document}
	
	
	
	\maketitle
	
	\begin{abstract}
	\noindent We give formulas for the density of visible points of several families of planar quasicrystals, which include the Ammann--Beenker point set and vertex sets of some rhombic Penrose tilings. These densities are used in order to calculate the limiting minimal normalised gap between the angles of visible points in two families of planar quasicrystals, which include the Ammann--Beenker point set and vertex sets of some rhombic Penrose tilings.
	\end{abstract}
	
	\section{Introduction}
	\label{secIntro}
	Given a locally finite point set $\mc{P}\subset \R^d$, let $\widehat{\mathcal{P}}=\{x\in \mathcal{P}\mid tx\notin \mathcal{P}, \forall t\in (0,1)\}$ denote the subset of points  that are \textit{visible} from the origin. If $\mc{P}\subset \R^2$, then within each finite horizon $T>0$, an observer located at the origin will see points in a finite number of directions, which correspond to the arguments of the visible points within $\widehat{\mc{P}}_T:=\widehat{\mc{P}}\cap B_T(0)$, where $B_T(0)=\{x\in \R^2: |x|<T\}$. For a large family of point sets $\mc{P}\subset \R^2$, including regular cut-and-project sets, the directions of visible points in $\widehat{\mc{P}}_T$ become uniformly distributed in $(-\pi,\pi]$ as $T\to\infty$. In this paper, we consider the fine-scale statistics of the distribution of visible points, i.e.\ the limiting distribution of normalised gaps between the angles of visible points in a locally finite point set $\mathcal{P}\subset \bb{R}^2$.
	
	For $T\in \bb{R}_{>0}$, let $\widehat{N}(T)=\#\widehat{\mathcal{P}}_T$. Let $-\pi<\alpha(x)\leq \pi$ denote the argument of $x\in \bb{R}^2$ viewed as a complex number and arrange $\frac{\alpha(x)}{2\pi}$, $x\in \#\widehat{\mathcal{P}}_T$, in increasing order as
	\begin{equation}
	\label{eqnXiHat}-\tfrac{1}{2}<\widehat{\xi}_{T,1}<\widehat{\xi}_{T,2}< \cdots<\widehat{\xi}_{T,\widehat{N}(T)}\leq \tfrac{1}{2},\end{equation}	
	Define also $\widehat{\xi}_{T,0}=\widehat{\xi}_{T,\widehat{N}(T)}-1$. Given an integer $1\leq i\leq \widehat{N}_T$, let $\widehat{d}_{T,i}=\widehat{N}(T)(\widehat{\xi}_{T,i}-\widehat{\xi}_{T,i-1})$. We call $\widehat{d}_{T,i}$ a \textit{normalised gap} (between the angles of visible points) in $\mc{P}$. Let also \[\widehat{\delta}_T=\underset{1\leq i \leq\widehat{N}(T)}{\min}\widehat{d}_{T,i}.\] 
	Form the probability measure 
	\[\mu_T=\frac{1}{\widehat{N}_T}\sum_{i=1}^{\widehat{N}_T}\delta_{\widehat{d}_{T,i}},\]
	where $\delta_x$ is the Dirac measure of $x\in \R$. Let $F_T:\R\longrightarrow [0,1]$ be the complementary distribution function of $\mu_T$, that is
	\begin{equation}
	\label{eqnLimMinGap}
	F_T(s)=\frac{\#(1\leq i\leq \widehat{N}_T\mid \widehat{d}_{T,i}\geq s)}{\widehat{N}_T}=\mu_T([s,\infty)).
	\end{equation} If $\mu_T$ converges weakly to a Borel probability measure $\mu$ on $\R$, or equivalently, $F_T(s)$ converges to $F(s):=\mu([s,\infty))$ at all continuity points of $F$, we say that the limiting distribution of normalised gaps between the angles of visible points exists. In this case, we call $F$ \textit{the limiting distribution of normalised gaps in} $\mc{P}$. A natural question is to determine for which point sets $\mc{P}$ the measure $\mu$ and the corresponding limiting distribution $F$ exists.
	
	In \cite{boca2000distribution}, Boca, Cobeli and Zaharescu proved that the limiting distribution of minimal gaps $F$ exists as a continuous function in the case $\mc{P}=\bb{Z}^2$, and gave the following explicit formula
	\begin{equation}
	\label{eqnLimGapDistrZ2}
	-F'(s)=\begin{cases}
	0, & s\leq \tfrac{3}{\pi^2},\\
	\tfrac{6}{\pi^2s^2}\cdot \log\tfrac{\pi^2s}{3}, & \tfrac{3}{\pi^2}\leq s\leq \tfrac{12}{\pi^2},\\
	\tfrac{12}{\pi^2s^2}\cdot \log\left(2/\left(1+\sqrt{1-\tfrac{12}{\pi^2s}}\right)\right), & \tfrac{12}{\pi^2}\leq s.
	\end{cases}\end{equation}
	In particular, they proved
	the existence of a \textit{minimal gap} in the limit, i.e.\ that there is some $m_{\widehat{\mc{P}}}>0$ with
	\begin{equation}
	\label{defnMPhat1}
	m_{\widehat{\mathcal{P}}}=\sup\{\sigma\geq 0\mid F(s)=1~\mathrm{for~ all}~s\in[0,\sigma]\},
	\end{equation} 
	which can be interpreted as a repulsion among directions of visible points.
	By the explicit expression for $F(s)$ given in \eqref{eqnLimGapDistrZ2} it follows that $m_{\widehat{\Z^2}}=\frac{3}{\pi^2}$. 
	
	In \cite{marklof2010distribution}, Marklof and Strömbergsson studied the fine-scale distribution of the directions of points in affine lattices of arbitrary dimension and characterised the distributions in terms of probability measures on associated homogeneous spaces. In particular, their result \cite[Corollary 2.7]{marklof2010distribution} implies that the limiting distribution of minimal gaps $F(s)$ exists continuously when $\mc{P}\subset\R^2$ is an affine lattice. In \cite{baake2014radial}, Baake, Götze, Huck and Jakobi numerically computed the  normalised gaps $\widehat{d}_{T,i}$ for large $T$, in prominent examples of planar quasicrystals such as the Ammann--Beenker point set (see \Cref{fig:AB} below) and the Tübingen triangle tiling. These gaps were then distributed in histograms (cf.\ \Cref{fig:Histogram} below) which were compared to the analytic expression for the limiting distribution of minimal gaps for $\Z^2$ in \eqref{eqnLimGapDistrZ2}.
	
		\begin{figure}[H]
		\centering
		\caption{A part of the Ammann--Beenker point set.}
		\begin{subfigure}{.5\textwidth}
			\centering
			\includegraphics[width=0.6\linewidth]{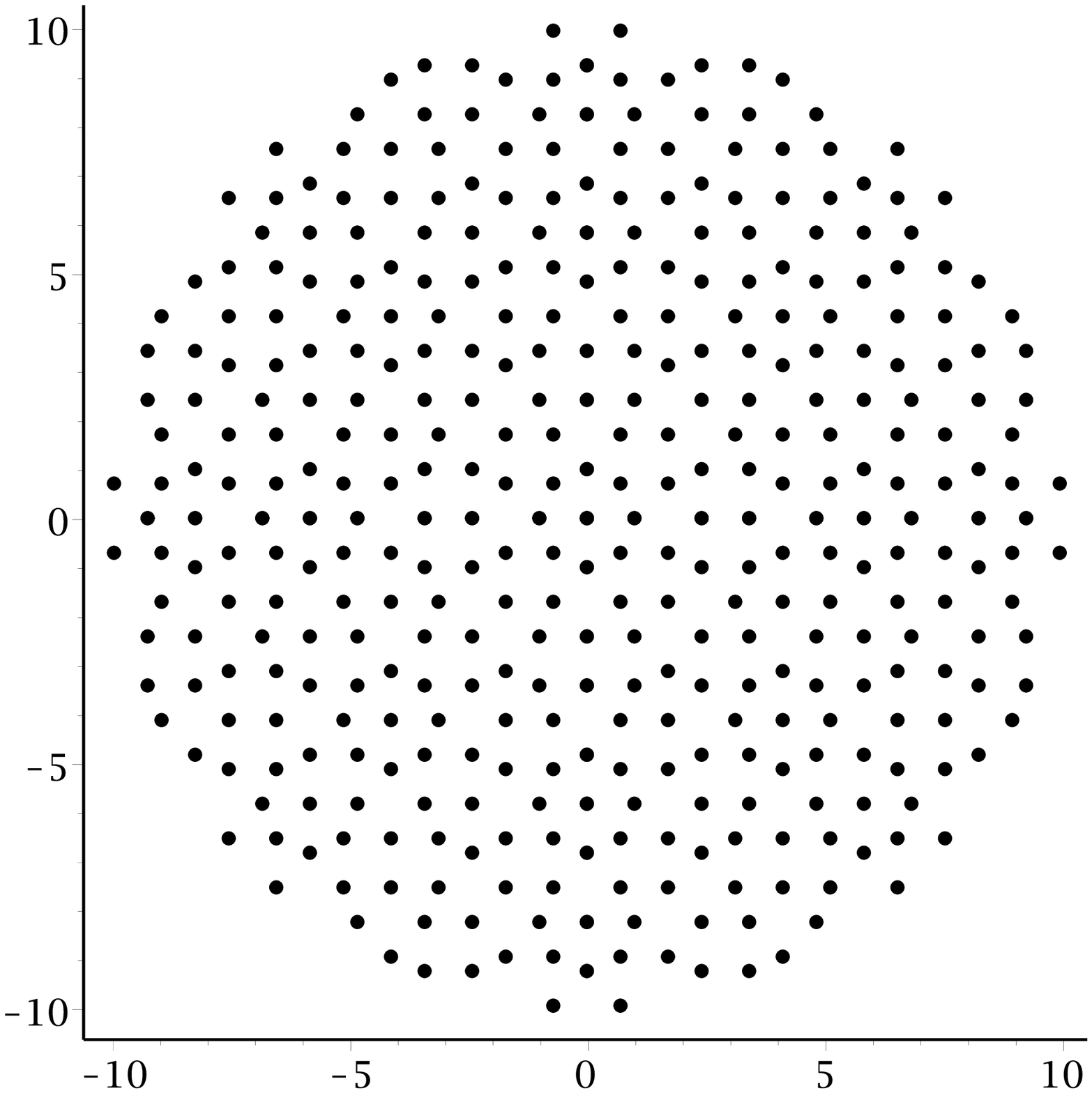}
			\caption{All points in $B_{10}(0)$.}
			\label{fig:ABsub1}
		\end{subfigure}%
		\begin{subfigure}{.5\textwidth}
			\centering
			\includegraphics[width=0.6\linewidth]{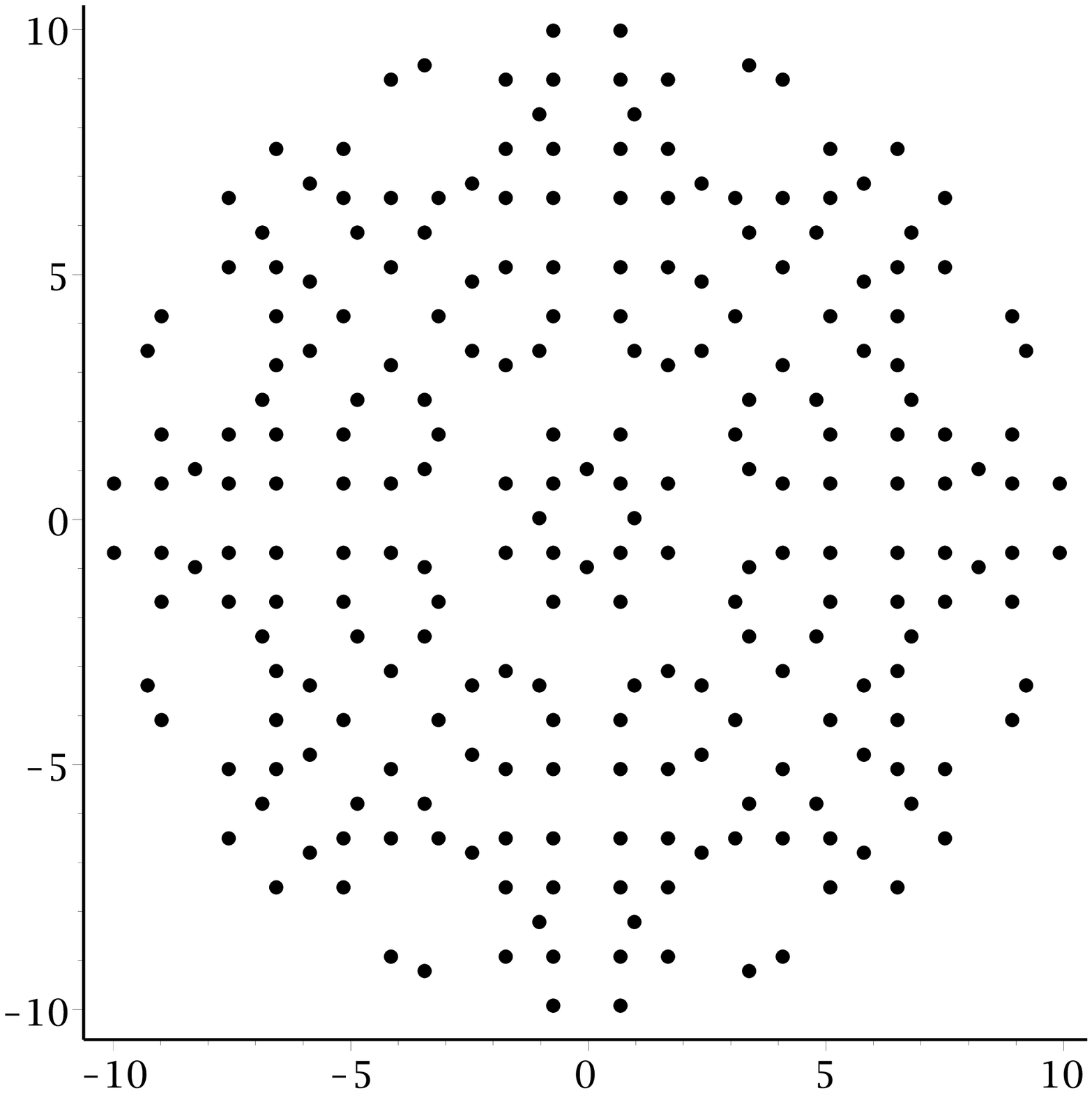}
			\caption{All visible points in $B_{10}(0)$.}
			\label{fig:ABsub2}
		\end{subfigure}
		\label{fig:AB}
	\end{figure} 

	Several gap distributions that were considered in \cite{baake2014radial} exhibited a minimal gap at a fixed, large radius, indicating the existence of a minimal gap in the limit. Furthermore, the shape of the histograms in \cite{baake2014radial} suggest that the limit distributions should exist continuously for the quasicrystals investigated.
	
	\begin{figure}[H]
		\centering
		\includegraphics[width=0.4\linewidth]{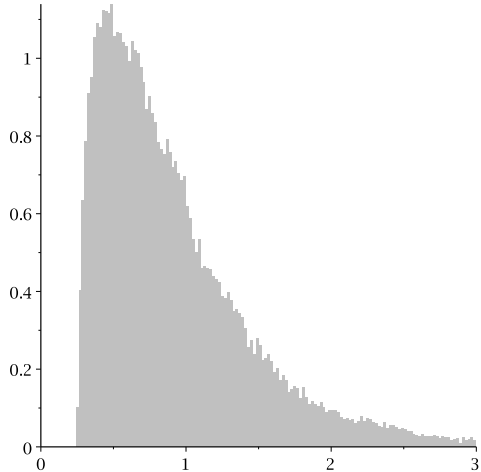}
		\caption{Statistics for the gaps of visible Ammann--Beenker points in $B_{700}(0)$, produced by distributing $\widehat{d}_{700,i}$, $i\in\{1,\ldots,\widehat{N}_{700}\}$, in bins of width $0.02$.}
		\label{fig:Histogram}
	\end{figure}
	
	In \cite[Corollary 3]{marklof2014visibility}, Marklof and Strömbergsson generalised the result from \cite{marklof2010distribution} mentioned above, by proving that for every regular planar cut-and-project set, the limiting distribution of normalised gaps exists as a continuous function, confirming some of the numerical observations in \cite{baake2014radial}. Furthermore, they expressed this limiting distribution explicitly in terms of a probability measure on an associated space of cut-and-project sets. In \cite{marklof2014visibility}, the existence of a positive minimal gap for several quasicrystals was also proved, again confirming numerical observations in \cite{baake2014radial}. For instance, Marklof and Strömbergsson proved the existence of a minimal gap in the Ammann--Beenker point set, as suggested by \Cref{fig:Histogram}.
	
	In this paper, we give formulas for the minimal gap between visible points in two families of quasicrystals, which include the Ammann--Beenker point set and vertex sets of some rhombic Penrose tilings. As we will see in \Cref{secCalcMPhat}, an important ingredient in the calculation of the minimal gap is the \textit{density of visible points} of a set. A locally finite point set $\mathcal{P}\subset \bb{R}^d$ is said to have an \textit{asymptotic density} (or simply \textit{density}) $\theta(\mathcal{P})$ if
	\[\lim_{T\to\infty}\frac{\#(\mathcal{P}\cap TD)}{\vol(TD)}=\theta(\mathcal{P})\]
	holds for all Jordan measurable $D\subset \bb{R}^d$ with $\vol(D)>0$. The density of visible points of a set $\mc{P}$ is thus $\theta(\widehat{\mc{P}})$. 
	
	It is well known that the density exists for a wide variety of point sets, in particular, the density of every regular cut-and-project set exists. In \cite[Theorem 1]{marklof2014visibility}, Marklof and Strömbergsson proved that the density of the subset of visible points of a regular cut-and-project set exists as well. However, the density of visible points of a set is only known explicitly in a few cases; we mention some of those here. For $d\geq 2$, we have $\widehat{\bb{Z}^d}=\{(n_1,\ldots,n_d)\in \bb{Z}^d\mid \gcd(n_1,\ldots,n_d)=1\}$, and the well known result $\theta(\widehat{\bb{Z}^d})=1/\zeta(d)$ gives the probability that $d$ random integers share no common factor. This can be derived in several ways, see for instance \cite{nymann1972probability}; we sketch another proof in \Cref{secDensZn}. More generally, $\theta(\widehat{\mathcal{L}})=\frac{1}{\vol(\bb{R}^d/\mathcal{L})\zeta(d)}$ for any lattice $\mathcal{L}\subset \bb{R}^d$, see e.g.\ \cite[Prop. 6]{baake2000diffraction}. In the presentation \cite{singppt1}, Sing computed the density of visible points in the Ammann--Beenker point set via an adelic approach. In this paper we prove \Cref{thmDensVisA}, which provides a formula for the density of visible points of a family of sets which includes the Ammann--Beenker point set. This result is then extended in \Cref{thmDensVisA2} to cover an even larger family of point sets. In particular, we recover Sing's result through another approach, whose general structure will be applicable to other families of point sets. For instance, we prove \Cref{thmDenVisPenrose}, which establishes a formula for the density of visible points for a family of rhombic Penrose tilings. We will then use these results to give formulas for the limiting minimal gaps in two families of quasicrystals, in \Cref{thmmPhatA} and \Cref{thmMPHatPenrose}, respectively. 

	This paper is organised as follows. First, in \Cref{secDensZn}, a proof of $\theta(\widehat{\bb{Z}^d})=1/\zeta(d)$ is given. In \Cref{secPointSets} the definition of a cut-and-project set is recalled and several families of quasicrystals obtained from the cut-and-project construction are presented. These families include the Ammann--Beenker point set and vertex sets of rhombic Penrose tilings. In \Cref{secDensVis} the density of visible points of sets from the above families are calculated and in \Cref{secCalcMPhat} these results are used to obtain the limiting minimal gap between the visible points for families of sets which include the Ammann--Beenker point set and vertex sets of rhombic Penrose tilings.
	
	\section{The density of visible points of $\bb{Z}^d$}
	\label{secDensZn}
		
	In this section we recall a proof of the well-known result $\theta(\widehat{\bb{Z}^d})=1/\zeta(d)$ for $d\geq 2$. The basic argument of the proof  will be used in later sections when calculating the density of visible points of other point sets.
		
	Fix $T>0$ and a Jordan measurable set $D\subset \bb{R}^d$, and let $\bb{P}\subset \bb{Z}_{>0}$ denote the set of prime numbers. For each \textit{invisible} point  $x\in \bb{Z}^d\setminus \widehat{\bb{Z}^d}$, there is some $p\in\bb{P}$ such that $\frac{x}{p}\in\bb{Z}^d$. For $\bb{Z}^d_*:=\bb{Z}^d\smpt{(0,\ldots,0)}$, there are only finitely many $p_1,\ldots,p_n\in\bb{P}$ such that $p_i\bb{Z}^d_*\cap TD\neq \emptyset$. By inclusion-exclusion counting we have
	\begin{align*}
		\#(\widehat{\bb{Z}^d}\cap TD)&=\#\left((\bb{Z}^d_*\cap TD)\setminus \bigcup_{p\in\bb{P}}(p\bb{Z}^d_*\cap TD)\right)=\#\left((\bb{Z}^d_*\cap TD)\setminus \bigcup_{i=1}^n(p_i\bb{Z}^d_*\cap TD)\right)\\
		&=\#(\bb{Z}^d_*\cap TD)+\sum_{k=1}^{n}(-1)^{k}\Big(\sum_{1\leq i_1<\ldots <i_k\leq n\label{\label{m}}}\#(p_{i_1}\bb{Z}^d_*\cap\cdots\cap p_{i_k}\bb{Z}^d_*\cap TD)\Big).
	\end{align*}
	The last sum can be rewritten as
	$\sum_{m=1}^\infty\mu(m)\#(m\bb{Z}^d_*\cap TD)$,
	where $\mu$ is the Möbius function. Hence
	\[\frac{\#(\widehat{\bb{Z}^d}\cap TD)}{\vol(TD)}=\sum_{m=1}^\infty\frac{\mu(m)\#(m\bb{Z}^d_*\cap TD)}{\vol(TD)}=\sum_{m=1}^\infty\frac{\mu(m)}{m^d}\frac{\#(\bb{Z}^d_*\cap \inv{m}TD)}{\vol(\inv{m}TD)}.\]
	Letting $T\to\infty$, switching order of limit and summation (for instance justified by finding a constant $c$ depending on $D$ such that $\#(\bb{Z}_*^d\cap TD)\leq c\cdot\vol(TD)$ for all $T$) and using $\theta(\bb{Z}^d_*)=1$, it follows that
	\[\theta(\widehat{\bb{Z}^d})=\lim_{T\to\infty}\frac{\#(\widehat{\bb{Z}^d}\cap TD)}{\vol(TD)}=\sum_{m=1}^\infty\frac{\mu(m)}{m^d}=\frac{1}{\zeta(d)}.\]

	\section{Particular families of point sets}
	\label{secPointSets}
	
	In this section we first recall the definition of a cut-and-project set and then introduce three families of such sets which we will consider throughout the remainder of the paper.
	
	\subsection{Cut-and-project sets}
	\label{secCPS}
	Cut-and-project sets are sometimes called (Euclidean) model sets. We will use the same notation and terminology for cut-and-project sets as in \cite[Sec. 1.2]{marklof2014free}. For an introduction to cut-and-project sets, see e.g.\ \cite[Ch. 7.2]{baake2013aperiodic}.

	If $\bb{R}^n=\bb{R}^d\times \bb{R}^m$, let 
	\begin{alignat*}{2}
	\pi:& ~\bb{R}^n\longrightarrow \bb{R}^d  & \pi_{\mathrm{int}}: & ~ \bb{R}^n\longrightarrow \bb{R}^m \\
	& (x_1,\ldots,x_n)\longmapsto (x_1,\ldots,x_d)\hspace{1cm}& & (x_1,\ldots,x_n)\longmapsto (x_{d+1},\ldots, x_n)
	\end{alignat*}
	denote projections onto $\R^d$ and $\R^m$ respectively.

	\begin{defn}
	\label{defnCPS}
	Let $\mathcal{L}\subset \bb{R}^n=\bb{R}^d\times \bb{R}^m$ be a lattice and $\mathcal{W}\subset \overline{\pi_{\mathrm{int}}(\mathcal{L})}$ be a bounded set with non-empty interior. Then the \textit{cut-and-project} set of $\mathcal{L}$ and $\mathcal{W}$ is given by
	\[\mathcal{P}(\mathcal{W},\mathcal{L})=\{\pi(y)\mid y\in \mathcal{L},\pi_{\mathrm{int}}(y)\in\mathcal{W}\}.\]
	\end{defn}

	The set $\mathcal{P}(\mathcal{W},\mathcal{L})$ is uniformly discrete since $\mathcal{W}$ is bounded and relatively dense since $\mathcal{W}^\circ$ is non-empty (cf.\ \cite[Prop. 3.1]{marklof2014free}); hence $\mathcal{P}(\mathcal{W},\mathcal{L})$ is \textit{Delone}. If $\partial \mc{W}$ has measure zero with respect to Haar measure on $\overline{\pi_{\mathrm{int}}(\mathcal{L})}$ we say that $\mathcal{P}(\mathcal{W},\mathcal{L})$ is \textit{regular}. If $\mathcal{L}$ is an affine lattice, i.e.\ $\mc{L}=\mc{L}_0+x$ for some lattice $\mc{L}_0\subset \R^n$ and some $x\in \R^n$, we extend the above definition by letting $\mc{P}(\mc{W},\mc{L})=\mc{P}(\mc{W}-\pi_{\mathrm{int}}(x),\mc{L}_0)+\pi(x)$.
	From \cite[Prop. 3.2]{marklof2014free} we have the following.

	\begin{prop}
	\label{propDensity}
	Let $\bb{R}^n=\bb{R}^d\times \bb{R}^m$ and let $\mathcal{P}=\mathcal{P}(\mathcal{W},\mathcal{L})$ be a regular cut-and-project set such that $\pi\mid_{\mathcal{L}}$ is injective and $\pi_\mathrm{int}(\mc{L})$ is dense in $\R^m$. Then the density $\theta(\mathcal{P})$ exists and is equal to $\frac{\vol(\mathcal{W})}{\vol(\bb{R}^n/\mathcal{L})}$. 
	\end{prop}	

	\subsection{$\mathcal{A}$-sets and $\mathcal{T}$-sets}
	\label{secATsets}
	
	Given $n\geq 2$, let $\zeta=e^{\frac{2\pi i}{n}}$ be an $n$-th root of unity. For $1\leq i\leq n-1$ with $\gcd(i,n)=1$, let $\sigma_i$ be the automorphism of the cyclotomic field $\bb{Q}(\zeta)$ determined by $\sigma_i(\zeta)=\zeta^i$.
	
	Let $n=8$ and $\sigma=\sigma_3$. Note that $\sigma$ induces the non-trivial automorphism of $\Q(\z)\cap \R=\Q(\sqrt{2})$. For a bounded set $\mathcal{W}\subset \bb{C}$, let \[\mathcal{A}_\mathcal{W}=\{x\in\bb{Z}[\zeta]\mid \sigma(x)\in\mathcal{W}\}\]
	and call this an $\mathcal{A}$-set. When $\mathcal{W}$ is the open regular octagon with side length $1$ centered at the origin with sides perpendicularly bisected by the coordinate axes, $\mathcal{A}$ is \textit{the Ammann--Beenker point set}, see \cite[Example 7.8]{baake2013aperiodic}. This set can also be realised as the vertices of a substitution tiling, see \cite[Ch. 6.1 and p. 236]{baake2013aperiodic}. Let \[\mathcal{L}=\{(x_1,x_2,\sigma(x_1),\sigma(x_2))\mid (x_1,x_2)\in\bb{Z}[\sqrt{2}]^2\}\subset\bb{R}^4\] be the Minkowski embedding of $(\bb{Z}[\zeta]\cap\bb{R})^2=\bb{Z}[\sqrt{2}]^2$ in $\bb{R}^4$. Straightforward calculations show that $\mathcal{A}_\mathcal{W}\subset \bb{C}$ can be identified\footnote{Throughout this paper we will frequently identify $\R^2$ and $\C$ in the natural way.} with 
	\begin{equation}
	\label{eqnACPS}
	\mathcal{P}(\mathcal{W}\inv{A},\mathcal{L}B)=\mathcal{P}(\mathcal{W}\inv{A},\mathcal{L})A_1\subset\bb{R}^2,
	\end{equation}
	where \[A=\begin{pmatrix}1 & 0\\[2pt] -\frac{1}{\sqrt{2}} & \frac{1}{\sqrt{2}}\end{pmatrix}\in\GL{2}{\bb{R}}, \quad B=\begin{pmatrix}
	A_1 & 0\\
	0 & I_2
	\end{pmatrix}\in\GL{4}{\bb{R}},\quad A_1=\begin{pmatrix}1 & 0\\[2pt] \frac{1}{\sqrt{2}}& \frac{1}{\sqrt{2}}\end{pmatrix}\] and $I_2$ is the identity matrix.
	This is a cut-and-project set in the sense of \Cref{defnCPS}.

	Let now $n=5$ and $\sigma=\sigma_2$. Note that $\sigma$ induces the non-trivial automorphism of $\Q(\z)\cap \R$. For a bounded set $\mathcal{W}\subset \bb{C}$, let \[\mathcal{T}_\mathcal{W}=\{x\in\bb{Z}[\zeta]\mid \sigma(x)\in\mathcal{W}\}\]
	and call this a $\mathcal{T}$-set. As above, let $\mathcal{L}$ be the Minkowski embedding of $(\bb{Z}[\zeta]\cap\bb{R})^2=\bb{Z}[\tau]^2$ in $\bb{R}^4$, where $\tau=\frac{1+\sqrt{5}}{2}$ is the golden ratio. Then $\mathcal{T}_\mathcal{W}\subset \bb{C}$ can be identified with 	
	\begin{equation}
		\label{eqnTCPS}
		\mathcal{P}(\mathcal{W}\inv{A},\mathcal{L}B)=\mathcal{P}(\mathcal{W}\inv{A},\mathcal{L})A_1\subset\bb{R}^2,
	\end{equation}
	where \[A=\begin{pmatrix}1 & 0\\[2pt]-\frac{\tau}{2} & \frac{\sqrt{\tau+2}}{2\t}\end{pmatrix},\quad B=\begin{pmatrix}
	A_1 & 0\\
	0 & I_2
	\end{pmatrix}\in\GL{4}{\bb{R}},\quad A_1=\begin{pmatrix}1 & 0\\[2pt]
	\frac{1}{2\tau} & \frac{\sqrt{\tau+2}}{2}\end{pmatrix},\] hence $\mathcal{T}_\mathcal{W}$ is a cut-and-project set according to \Cref{defnCPS}. In \cite[Section 4]{baake1990planar}, substitution tilings and corresponding vertex sets are obtained using two triangular tiles. If $\mathcal{W}$ is the closed regular decagon of side length $\sqrt{\frac{\t+2}{5}}$ centered at the origin with two vertices at the $y$-axis, then from e.g.\ \cite[(4.3)]{baake1990planar} one can verify that for almost all $\e\in \C$ the set $\mc{T}_{\mc{W}+\e}$ is the vertex set of such a triangular tiling. In particular, this holds for $\epsilon=0$ which gives a point set with fivefold rotational symmetry, see \cite[Fig. 4.4]{baake1990planar}.
	
	\subsection{$\mathcal{P}$-sets}
	\label{secPsets}
	
	Again, let $n=5$ and $\sigma=\sigma_2$. Let also $\kappa:\bb{Z}[\zeta]\longrightarrow \bb{Z}/5\bb{Z}$ be the ring homomorphism determined by $\kappa(\z)=1$. The kernel of this map is the prime ideal $(1-\z)$ of $\Z[\z]$ generated by $1-\zeta$. Let $\mathcal{W}_1$ be the interior of the convex hull of $\{1,\zeta,\zeta^2,\zeta^3,\zeta^4\}$ in $\C$, an open regular pentagon centered at the origin. Given $\epsilon\in\bb{C}$, let $\mathcal{W}_{1,\epsilon}=\mathcal{W}_1+\epsilon$, $\mathcal{W}_{2,\epsilon}=-\tau\mathcal{W}_1+\epsilon$, $\mathcal{W}_{3,\epsilon}=\tau\mathcal{W}_1+\epsilon$ and $\mathcal{W}_{4,\epsilon}=-\mathcal{W}_1+\epsilon$. Following \cite[Example 7.11]{baake2013aperiodic}, define for $k\in\{1,2,3,4\}$
	\[\Lambda_{k,\epsilon}=\{x\in\bb{Z}[\zeta]\mid \kappa(x)=k,\sigma(x)\in\mathcal{W}_{k,\e}\}\]
	and then define
	\begin{equation}
	\label{eqnDefPenroseEps}
		\mathcal{P}_{\epsilon}=\bigcup_{k=1}^4\Lambda_{k,\epsilon}.
	\end{equation}
	We will call $\mathcal{P}_\e$ a $\mathcal{P}$-set.
	
	For $j\in\{0,\ldots,4\}$, let $v_j=\sqrt{\frac{2}{5}}(\cos(\frac{2\pi j}{5}),\sin(\frac{2\pi j}{5}),\cos(\frac{4\pi j}{5}),\sin(\frac{4\pi j}{5}),\frac{1}{\sqrt{2}})$ and let $g\in \mathrm{SO}(5,\R)$ be the matrix whose $(j+1)$-th row is $v_j$ and let $\mathcal{L}=\Z^5g$.  Let \[\mathcal{W}_\e=\bigcup_{k=1}^4\sqrt{\tfrac{2}{5}}\mathcal{W}_{k,\epsilon}\times \left\{\tfrac{k}{\sqrt{5}}\right\}\subset \R^3.\] Let $\pi$, $\pi_{\mathrm{int}}$ denote the projections from $\R^5$ onto the first two and last three coordinates, respectively. As shown in \cite{marklof2014free}, we have $\overline{\pi_{\mathrm{int}}(\mc{L})}=\R^2\times \frac{1}{\sqrt{5}}\bb{Z}$; note that $\mathcal{W}_\e\subset \overline{\pi_{\mathrm{int}}(\mc{L})}$ is open. Consider now the regular cut-and-project set 
	\begin{equation}
	\label{eqnDefPenroseEps'}\mc{P}_\e':=\mathcal{P}(\mathcal{W}_\e,\mathcal{L}).
	\end{equation} 
	We claim that $\sqrt{\frac{2}{5}}\mathcal{P}_\e=\mc{P}_\e'$.
	
	Indeed, note that $\mc{P}_\e'=\mathcal{P}(\mathcal{W}_\e,\mathcal{L})$ consists of elements of the form $\pi(kg)$ with $k\in \Z^5$ such that $\pi_{\mathrm{int}}(kg)\in\mathcal{W}_\e$. We can identify $\pi(kg)\in \R^5$ with $\sqrt{\frac{2}{5}}x=\sum_{j=0}^4k_j\z^j\in\C$ and $\pi_{\mathrm{int}}(kg)$ with $\left(\sqrt{\frac{2}{5}}\sigma(x),\frac{\sum_{j=0}^4k_j}{\sqrt{5}}\right)$. The claim follows by noting that every $x=\sum_{j=0}^4k_j\z^j\in \Z[\z]$ can be modified so that $\sum_{j=0}^4k_j\in\{0,1,2,3,4\}$ since $1+\z+\z^2+\z^3+\z^4=0$.
	
	A combination of \cite[Theorems 8.1, 11.1]{de1981algebraic} gives the following.
	
	\begin{thm}
		\label{thmdeBruijn}
		Let $\e\in \C$. Then  $\mathcal{P}_{\epsilon}$ is the vertex set of a rhombic Penrose tiling if $\epsilon=\sum_{j=0}^{4}\gamma_j\zeta^{2j}$ for some $\gamma\in\bb{R}^5$ with $\sum_{j=0}^{4}\gamma_j=0$ such that $\epsilon\notin\bigcup_{k=0}^4( \bb{R}\zeta^ki+(1-\zeta))$.
	\end{thm} 

	In \Cref{lemExOfRegularPenta} we verify that \Cref{thmdeBruijn} holds for all $\gamma\in(\bb{Q}\setminus \bb{Z})^5$ with $\sum_{j=0}^{4}\gamma_j=0$. 
	
	\section{Calculation of densities of visible points}
	\label{secDensVis}
	
	In this section we calculate the density of visible points of families of $\mathcal{A}$-, $\mathcal{T}$- and $\mathcal{P}$-sets after presenting some auxiliary results. Firstly, we have the following lemma, which is immediate from the definitions.
	
	\begin{lem}
	\label{lemDensitiesGeneral}
	Suppose $\mathcal{P}\subset \bb{R}^d$ is locally finite and that $\theta(\mathcal{P})$ exists. Then for any $A\in \GL{d}{\bb{R}}$ we have
	$\widehat{\mathcal{P}A}=\widehat{\mathcal{P}}A$ and $\theta(\mathcal{P}A)=\frac{\theta(\mathcal{P})}{\det(A)}$.
	\end{lem}

	Next, we prove that the density of visible points is unaffected when passing to a subset of full density.
	
	\begin{lem}
	\label{lemDensitiesGeneral2}
	If $\mathcal{P}_1\subset \mathcal{P}_2\subset\bb{R}^d$ are locally finite with $\theta(\mathcal{P}_1)=\theta(\mathcal{P}_2)$ and if $\theta(\widehat{\mathcal{P}_1})$ and  $\theta(\widehat{\mathcal{P}_2})$ both exist, then $\theta(\widehat{\mathcal{P}_1})=\theta(\widehat{\mathcal{P}_2})$.
	\end{lem}
	
	\begin{proof} 
	We have $\mathcal{P}_1\setminus\widehat{\mathcal{P}_1}\subset \mathcal{P}_2\setminus\widehat{\mathcal{P}_2}$ and it suffices to show that
	$\theta((\mathcal{P}_2\setminus\widehat{\mathcal{P}_2})\setminus(\mathcal{P}_1\setminus\widehat{\mathcal{P}_1}))=0.$
	To this end, let $D$ be a Jordan measurable set and let $T>0$ be given. Consider the set  \[D_T:=((\mathcal{P}_2\setminus\widehat{\mathcal{P}_2})\setminus(\mathcal{P}_1\setminus\widehat{\mathcal{P}_1}))\cap TD.\]
	We say that $x\in D_T$ is of type 1 if $x\in\mathcal{P}_1\cap TD$. Then, $x\in \widehat{\mathcal{P}_1}\cap TD$ so $x$ there must be $\alpha>1$ with $x/\alpha\in (\mathcal{P}_2\setminus \mathcal{P}_1)\cap TD$; let $\alpha(x)$ be the minimal such $\alpha$. Otherwise we say that $x$ is of type 2. Define a map $f: D_T\longrightarrow ((\mathcal{P}_2/\mathcal{P}_1)\cap TD)\times \{1,2\}$ by $f(x)=(x/\alpha(x),1)$ if $x$ is of type $1$ and $x\mapsto (x,2)$ if $x$ is of type 2. We claim that this map is injective. Indeed, suppose that $(x/\alpha(x),1)=(y/\alpha(y),1)$. Then if $x\neq y$ we can assume that there is $k\in(0,1)$ with $kx=y$, which contradicts $x\in \widehat{\mathcal{P}_1}$ since $y\in\mathcal{P}_1$. 
	
	 Since $\theta(\mathcal{P}_1)=\theta(\mathcal{P}_2)$, it follows that
	$\lim_{T\to\infty}\frac{\#D_T}{\vol(TD)}\leq \lim_{T\to\infty}\frac{2\#((\mathcal{P}_2\setminus\mathcal{P}_1)\cap TD)}{\vol(TD)}=0.$
	\end{proof}

	Note that if $\mathcal{P}=\mathcal{P}(\mathcal{W},\mathcal{L})$ is a regular cut-and-project set, then with $\mathcal{P}_1=\mathcal{P}(\overline{\mathcal{W}},\mathcal{L})$ we have $\theta(\widehat{\mc{P}})=\theta(\widehat{\mc{P}_1})$ by \cite[(2.4)]{marklof2014visibility} and the last part of \Cref{lemDensitiesGeneral2}, since we know that $\theta(\widehat{\mc{P}})$, $\theta(\widehat{\mc{P}_1})$ both exist by \cite[Theorem 1]{marklof2014visibility}.

	It can be shown that for a locally finite point set $\mc{P}$ in $\R^2$, the subset of visible points of Lebesgue almost every translate of $\mc{P}$ has full density. In \cite{marklof2010distribution}, it is observed that if $\Z^2_\a:=\alpha+\Z^2$ contains invisible points on two distinct lines through the origin, then $\alpha\in \Q^2$, which implies that $\theta(\widehat{\Z^2_\a})=\theta(\Z^2)$ for all $\alpha\in\bb{R}^2\setminus \Q^2$. Thus $\widehat{\Z^2_\a}$ does not have full density only if $\alpha\in\Q^2$. The following result is similar.

	\begin{prop}
		\label{propDensVisGeneric} 	Let $K$ be a real number field. Let $\mathcal{P}\subset K^2$ and $\alpha\in\bb{R}^2$ be given. If there is a line through the origin that contains two points of $\alpha+\mathcal{P}$ then $\dim\mathrm{span}_K\{1,\alpha_1,\alpha_2\}\leq 2$. If there are two distinct lines through the origin that contain two points of $\alpha+\mc{P}$ then $\alpha\in K^2$.
	\end{prop}

	\begin{proof}
		Suppose there is a line $\ell$ through the origin that contains two distinct points of $\alpha+\mathcal{P}$, say $m+\alpha$, $n+\alpha$ for some $m,n\in K^2$ with $m\neq n$. Then there is some real $t\neq 1$ with $t(m+\alpha)=n+\alpha$ or equivalently $\frac{tm}{1-t}-\frac{n}{1-t}=\alpha$. Let $s=\frac{t}{1-t}$ so that $1+s=\frac{1}{1-t}$. Thus $sm-(1+s)n=\alpha$. Since $m\neq n$ there is $i$ such that $m_i\neq n_i$. We then have $s=\frac{\alpha_i+n_i}{m_i-n_i}$ and $1+s=\frac{\alpha_i+m_i}{m_i-n_i}$ and we see that $\alpha_j\in \mathrm{span}_K\{1,\alpha_i\}$. The first claim is thus proved.
		
		Suppose now there are distinct lines $\ell_1$, $\ell_2$ through the origin that contain two points of $\alpha+\mathcal{P}$. Thus, there are $m_1,m_2,n_1,n_2\in K^2$ and $t_1,t_2\neq 1$ such that $t_i(m_i+\alpha)=n_i+\alpha$. Since $\ell_i$ are distinct the direction vectors $m_i-n_i$ are not proportional. As above, we have $\alpha=s_im_i-(1+s_i)n_i=s_i(m_i-n_i)-n_i$ for some real numbers $s_i$. Hence we get the following system of equations $s_1(m_1-n_1)-n_1=s_2(m_2-n_2)-n_2$. Since $m_1-n_1$, $m_2-n_2$ are not proportional the system has a unique solution, which has to belong to $K^2$. Hence $\alpha\in K^2$.
	\end{proof}

	It follows that if $\mathcal{P}\subset K^2$ and $\alpha \in \R^2\setminus K^2$, then $\theta(\widehat{\a+\mathcal{P}})=\theta(\mathcal{P})$. This result can be applied to $\mathcal{A}$-sets with $K=\Q(\sqrt{2})$ and $\mathcal{T}$-, $\mathcal{P}$-sets with $K=\Q(\sqrt{\t+2})$, by \eqref{eqnACPS}, \eqref{eqnTCPS}.
	
	Given a locally finite point set $\mathcal{P}\subset \bb{R}^d$, we call $C\subset \R_{>1}$ a \textit{set of occlusion quotients for} $\mc{P}$ if for each $x\in \mc{P}\setminus \widehat{\mc{P}}$ there exists $c\in C$ with $x/c\in\mc{P}$. Note that each locally finite point set $\mc{P}$ has a set of occlusion quotients. Let also $\mathcal{P}_*=\mathcal{P}\setminus\{(0,\ldots,0)\}$. Next, a counting formula for the number of visible points in a bounded set is presented. 	
	\begin{lem}
		\label{lemInclExcl}
		Let $\mathcal{P}\subset \bb{R}^d$ be locally finite and fix a set $C$ of occlusion quotients for $\mc{P}$. Let $T>0$ and a bounded set $D\subset \bb{R}^d$ be given. Then there are only finitely many $c\in C$ such that $\mathcal{P}_*\cap c\mathcal{P}_*\cap TD\neq\emptyset$, and
		\[\#(\widehat{\mathcal{P}}\cap TD)=\sum_{\substack{F\subset C\\\#F<\infty}}(-1)^{\#F}\#\Big(\Big(\mathcal{P}_*\cap \bigcap_{c\in F}c\mathcal{P}_*\Big)\cap TD\Big)\]
		(here the sum ranges over all finite subsets $F$ of $C$; in particular, $F=\emptyset$ gives the term $\#(\mathcal{P}_*\cap TD)$).
	\end{lem}
		
	\begin{proof}
		We first claim that the set $C_T:=\{c\in C\mid \mathcal{P}_*\cap c\mathcal{P}_*\cap TD\neq \emptyset\}$ is finite. Indeed, suppose this is not true and pick distinct $c_1,c_2,\ldots\in C_T$ and corresponding $x_i\in \mathcal{P}_*\cap c_i\mathcal{P}_*\cap TD$. Since $\mathcal{P}$ is locally finite, the sequence $x_1,x_2,\ldots$ contains only finitely many distinct elements. Thus, a subsequence $x_{k_1},x_{k_2},\ldots$ which is constant can be extracted, so that $x_{k_i}/c_{k_i}\in \mathcal{P}_*\cap \frac{TD}{c_{k_i}}\subset \mathcal{P}_*\cap B$ are all distinct, contradicting the assumption that $\mathcal{P}$ is locally finite. Here $B$ is some ball centered at $0$ with $TD\subset B$. Thus, we can write $C_T=\{c_1,\ldots,c_n\}$ for some $c_1,\ldots,c_n\in C$. Consequently
		\begin{align*}
			\#(\widehat{\mathcal{P}}\cap TD)&=\#\left((\mathcal{P}_*\cap TD)\setminus\bigcup_{c\in C}(\mathcal{P}_*\cap c\mathcal{P}_*\cap TD)\right)\\
			&=\#(\mathcal{P}_*\cap TD)-\#\left(\bigcup_{i=1}^n(\mathcal{P}_*\cap c_i\mathcal{P}_*\cap TD)\right),
		\end{align*}
		whence the result follows from the inclusion-exclusion counting formula for finite unions of finite sets.
	\end{proof}	
	
	\begin{lem}
		\label{lemLatticeBoxBound}
		For every lattice $\mathcal{L}\subset \bb{R}^d$ and $c>0$, there is a constant $L$ such that if $B=\prod_{i=1}^{d}[a_i,b_i]$ is a box with $b_i-a_i\geq c$ for all $1\leq i\leq d$, then $\#(B\cap\mathcal{L})\leq L\vol(B)$.
	\end{lem}
				
	\begin{proof}
		Let $n_i=\ceil{\frac{b_i-a_i}{c}}\in\bb{Z}_+$. Then
		$\frac{b_i-a_i}{c}\leq n_i< \frac{b_i-a_i}{c}+1=\frac{b_i-a_i+c}{c}<\frac{2(b_i-a_i)}{c}.$
		With $n=\prod_{i=1}^dn_i$ it follows that $n\leq \frac{2^d\vol(B)}{c^d}$. Also, $B$ can be covered by $n$ translates of $[0,c]^d$. Find now $L_1>0$, depending on $\mathcal{L}$ and $c$, such that $\sup_{t\in\bb{R}^d}\#(\mathcal{L}\cap (t+[0,c]^d))=L_1$. Hence
		$\#(B\cap \mathcal{L})\leq n L_1\leq \frac{2^dL_1}{c^d}\vol(B)$, so one can take $L=\frac{2^dL_1}{c^d}$.
	\end{proof}

	Given a real quadratic extension $K$ of $\bb{Q}$, let $\sigma$ denote the non-trivial automorphism of $K$ and let $\sigma(x)=(\sigma(x_1),\ldots,\sigma(x_n))$ for $x=(x_1,\ldots,x_n)\in K^n$. Let $N(x)=x\sigma(x)$ denote the norm of $x\in K$. Let $\epsilon>1$ be the fundamental unit of $\mathcal{O}_K$ and let \[\bb{P}=\{\pi\in\mathcal{O}_K\mid \pi\text{ prime},1<\pi<\epsilon\}.\] Given $y\in\mathcal{O}_K\smpt{0}$, let $\mathcal{L}_y=\{(x,\sigma(x))\mid x\in y\mathcal{O}_K^2\}\subset \bb{R}^4$. Note that for any unit $u\in \mc{O}_K$, $\mathcal{L}_u$ is the Minkowski embedding of $\mathcal{O}_K^2$ in $\bb{R}^4$. Let also $\mathcal{L}'_y=\{(x,\sigma(x))\mid x\in y\mathcal{O}_K\}\subset \bb{R}^2$. 
	
	Let $\mathfrak{I}$ be the set of ideals of $\mathcal{O}_K$ and $\mathfrak{P}\subset \mathfrak{I}$ be the subset of prime ideals. For $I\in\mathfrak{I}$, let $N(I)=\#(\mathcal{O}_K/I)$. For $s\in\bb{C}$ with $\mathrm{Re}(s)>1$, Dedekind's zeta function over $K$ is given by
	\[\zeta_K(s)=\sum_{I\in\mathfrak{I}}\frac{1}{N(I)^s}=\prod_{P\in\mathfrak{P}}\left(1-\frac{1}{N(P)^s}\right)^{-1}.\]

	When $\mathcal{O}_K$ is a unique factorisation domain (and hence also a principal ideal domain), write $\gcd(x,y)=1$ if $x,y$ are relatively prime. In this case, let also $\omega(I)$ be the number of distinct prime factors of any generator of $I$. Define a Möbius function $\mu:\mathfrak{I}\longrightarrow \bb{Z}$ by $\mu(I)=(-1)^{\omega(I)}$ if every generator of $I$ is square-free and $\mu(I)=0$ otherwise. By analogy with the Riemann zeta function we then have
	\begin{equation}
	\label{eqnDedekind}
	\frac{1}{\zeta_K(s)}=\sum_{I\in\mathfrak{I}}\frac{\mu(I)}{N(I)^s}.
	\end{equation}
	
	\begin{lem}
	\label{lemEstimate1}
	Given a real quadratic number field $K$ and bounded sets $\mathcal{W},D\subset \bb{R}^2$, with $\mc{W}$ star-shaped with respect to the origin, there is a constant $L$ such that for all $T>0$ and $y\in \mathcal{O}_K$ 
	\[\#(\mathcal{P}(\mathcal{W},\mathcal{L}_y)_*\cap TD)\leq \frac{LT^2}{N(y)^2}.\]
	\end{lem} 
			
	\begin{proof}
	As $\mathcal{L}_y=\mathcal{L}_{uy}$ for all units $u$, we may without loss of generality assume that $y>0$. Fix $r_D,r_\mathcal{W}>1$ such that $D\subset B_D:=[-r_D,r_D]^2$ and $\mathcal{W}\subset B_\mathcal{W}:=[-r_\mathcal{W},r_\mathcal{W}]^2$. There is a bijection \[\mathcal{P}(\mathcal{W},\mathcal{L}_y)_*\cap TD\longrightarrow\{x\in \mathcal{O}_K^2\mid \sigma(x)\in \tfrac{1}{\sigma(y)}\mathcal{W}\}_*\cap \tfrac{1}{y}TD\] given by $x\mapsto \tfrac{x}{y}$. Now, the right-hand set is in bijection with $\mathcal{L}_*\cap (\tfrac{1}{y}TD\times \tfrac{1}{\sigma(y)}\mathcal{W})$, by $x\mapsto(x,\sigma(x))$. It follows that
	\[\mathcal{P}(\mathcal{W},\mathcal{L}_y)_*\cap TD=\#(\mathcal{L}_*\cap (\tfrac{1}{y}TD\times \tfrac{1}{\sigma(y)}\mathcal{W})).\]
	Note also that the right-hand remains unchanged if $y$ is replaced by $uy$ for any unit $u\in\mc{O}_K$.
	
	Find now $c>0$ such that $c'\leq c$ implies that \[\mathcal{L}_*\cap (\e D\times c'(\mathcal{W}\cup(-\mathcal{W}))=\emptyset.\] This can be done, for otherwise $\mathcal{L}$ would contain elements of arbitrarily small non-zero fourth coordinate within the bounded set $\e D\times (\mathcal{W}\cup(-\mathcal{W}))$, contradicting that $\mathcal{L}$ is a lattice.
	
	Suppose first that $T$ satisfies $\left|\frac{T}{y\sigma(y)}\right|<c$. Scale $y$ by a positive unit such that $1\leq \frac{T}{y}<\e$. This implies that $\frac{1}{|\sigma(y)|}<c$. Then
	\[\mathcal{L}_*\cap (\tfrac{T}{y}D\times \tfrac{1}{\sigma(y)}\mathcal{W})\subset \mathcal{L}_*\cap (\e D\times c(\mathcal{W}\cup(-\mathcal{W})))=\emptyset,\]
	using the fact that $\mc{W}$ is star-shaped with respect to the origin.
	
	Suppose now $T$ satisfies $\left|\frac{T}{y\sigma(y)}\right|\geq c$. Scale $y$ by a positive unit so that $\sqrt{c}\leq \frac{T}{y}<\e\sqrt{c}$. This implies that $\frac{\sqrt{c}}{\e}\leq \frac{1}{|\sigma(y)|}$. It follows that $[0,\sqrt{c}]^4$ is contained in $B:=\frac{T}{y}B_D\times \frac{1}{\sigma(y)}B_\mathcal{W}$. By \Cref{lemLatticeBoxBound} there is a constant $L$, depending on $c$ and $\mathcal{L}$, such that
	\[\#(\mathcal{L}\cap B)\leq L\vol(B)=\frac{16Lr_D^2r_\mathcal{W}^2T^2}{N(y)^2}\]
	and $\mathcal{L}_*\cap (\tfrac{T}{y}D\times \tfrac{1}{\sigma(y)}\mathcal{W})\leq \#(\mathcal{L}\cap B)$.
	\end{proof}
	
	\subsection{$\theta(\widehat{\mathcal{A}_\mathcal{W}})$ for certain $\mathcal{W}$}
	\label{secdensAsets}
	
	Let $\z=e^{\frac{2\pi i}{8}}$, so that $K=\bb{Q}(\zeta)\cap \bb{R}=\bb{Q}(\sqrt{2})$, $\mathcal{O}_K=\bb{Z}[\sqrt{2}]$ and $\bb{Z}[\zeta]=\mathcal{O}_K\oplus \mathcal{O}_K\zeta$. Let $\sigma$ be the automorphism of $K$ given by $\z\mapsto\z^3$. Note that $\mathcal{O}_K$ is a Euclidean domain with fundamental unit $\lambda=1+\sqrt{2}$. Let $W_1$ denote the family of all Jordan measurable $\mathcal{W}\subset\bb{C}$ which are star-shaped with respect to the origin and satisfy $-\mathcal{W}\subset \sqrt{2}\mathcal{W}$.

	\begin{lem}
	\label{lemSizeOfPrimesA}
	For every $\pi\in\bb{P}$ we have $|\sigma(\pi)|\geq \sqrt{2}$.
	\end{lem}
	
	\begin{proof}
	Suppose towards a contradiction that there is a prime $\pi\in\bb{P}$ with $|\sigma(\pi)|<\sqrt{2}$. Then $(\pi,\sigma(\pi))\in \mathcal{L}_1'\cap((1,\lambda)\times (-\sqrt{2},\sqrt{2}))$, where the right-hand set is finite, being the intersection of a lattice and a bounded set, and can be verified to be empty by hand.
	\end{proof}
	
	The following proposition establishes visibility conditions in $\mathcal{A}_\mathcal{W}$ (recall the definition of $\mc{A}_\mc{W}$ in \eqref{eqnACPS}). Its statement in the special case of the Ammann--Beenker point set can be found in e.g.\ \cite[p. 427]{baake2013aperiodic}; a proof in this special case is given in \cite[Ch.\ 4]{jakobi2017radial}. Since our statement have weaker assumptions on the window $\mc{W}$ we write out a proof for clarity.
	
	\begin{prop}
	\label{propVisCondA}
	For $\mathcal{W}\in W_1$ we have
	\[\widehat{\mathcal{A}_\mathcal{W}}=\{x=x_1+x_2\zeta\in \mathcal{A}_\mathcal{W}\mid x_1,x_2\in\mc{O}_K,~\gcd(x_1,x_2)=1,~\sigma(x/\lambda)\notin \mathcal{W}\}.\]
	\end{prop}
	
	\begin{proof}
	We first prove that the visibility conditions are necessary. Suppose that $x=x_1+x_2\zeta\in \mathcal{A}_\mathcal{W}$ and that $x_1,x_2\in\bb{Z}[\sqrt{2}]$ are not relatively prime, i.e.\ there is a prime $\pi\in \bb{P}$ which divides $x_1,x_2$. Then $x/\pi\in \bb{Z}[\zeta]$. By \Cref{lemSizeOfPrimesA}, we have $\sigma(x/\pi)\in \frac{\mathcal{W}}{\sqrt{2}}\cup \left(\frac{-\mathcal{W}}{\sqrt{2}}\right)\subset \mathcal{W}$ since $\mathcal{W}\in W_1$ and hence $x/\pi\in\mathcal{A}_\mathcal{W}$. If $\sigma(x/\lambda)\in \mathcal{W}$, we have $x/\lambda\in \mathcal{A}_\mathcal{W}$. In either case, we have  $x\in\mathcal{A}_\mathcal{W}\setminus\widehat{\mathcal{A}_\mathcal{W}}$.
	
	For sufficiency, suppose $x=x_1+x_2\zeta\in \mathcal{A}_\mathcal{W}\setminus\widehat{\mathcal{A}_\mathcal{W}}$. Then there is some $\alpha>1$ with $x/\alpha\in \mathcal{A}_\mathcal{W}$, which implies that $\alpha\in \bb{Q}(\zeta)\cap \bb{R}=\bb{Q}(\sqrt{2})$. Since $\mathcal{A}_\mathcal{W}$ is locally finite, we may assume that $x/\alpha\in \widehat{\mathcal{A}_\mathcal{W}}$. Write $x/\alpha=y_1+y_2\zeta$ for some $y_1,y_2\in\bb{Z}[\sqrt{2}]$. By necessity, we must have $\gcd(y_1,y_2)=1$ which implies $\alpha\in \bb{Z}[\sqrt{2}]$. If $\alpha$ is not a unit, then $x_1,x_2$ are not relatively prime. Otherwise, $\alpha=\lambda^k$ for some $k\geq 1$. If $k=1$, then $\sigma(x/\lambda)\in\mathcal{W}$ and otherwise
	\[\sigma(x/\lambda^k)=(-1)^{k-1}\lambda^{k-1}\sigma(x/\lambda)\in\mathcal{W},\]
	which implies $\sigma(x/\lambda)\in\mathcal{W}$ since $\mathcal{W}\in W_1$.
	\end{proof}
	
	A consequence of \Cref{lemSizeOfPrimesA} and \Cref{propVisCondA} is that if $\mathcal{W}\in W_1$, then $ C:=\bb{P}\cup\{\lambda\}$ is a set of occlusion quotients for $\mathcal{A}_\mathcal{W}$. For a finite subset $F\subset C$ let $\Pi_F$ denote the product of the elements of $F$. It follows that
	\begin{equation}
	\label{eqnAW}
	\mathcal{A}_\mathcal{W}\cap\bigcap_{c\in F} c\mathcal{A}_\mathcal{W}=\begin{cases}
	\{x\in\bb{Z}[\zeta]\Pi_F\mid \sigma(x)\in\mathcal{W}\} \text{ if }\lambda\notin F,\\
		\{x\in\bb{Z}[\zeta]\Pi_F\mid \sigma(x)\in-\frac{1}{\lambda}\mathcal{W}\} \text{ if }\lambda\in F.
	\end{cases}\end{equation}
	It follows from \eqref{eqnACPS}, \Cref{propDensity} and \Cref{lemDensitiesGeneral} that
	\begin{equation}
	\label{eqnDensitiesIntersection}
	\theta\Big(\mathcal{A}_\mathcal{W}\cap\bigcap_{c\in F} c\mathcal{A}_\mathcal{W}\Big)=\begin{cases}
	\frac{\theta\left(\mathcal{A}_\mathcal{W}\right)}{N(\Pi_F)^2} \text{ if }\lambda\notin F,\\
		\frac{\theta\left(\mathcal{A}_\mathcal{W}\right)}{\lambda^2N(\Pi_F)^2} \text{ if }\lambda\in F.
	\end{cases}\end{equation}
	We are now ready to calculate $\theta(\widehat{\mathcal{A}_\mathcal{W}})$. 
	
	\begin{thm}
		\label{thmDensVisA}
		For $\mathcal{W}\in W_1$ we have
		\[\theta(\widehat{\mathcal{A}_\mathcal{W}})=\sum_{\substack{F\subset C\\\#F<\infty}}(-1)^{\#F}\theta\Big(\mathcal{A}_\mathcal{W}\cap\bigcap_{c\in F} c\mathcal{A}_\mathcal{W}\Big)=\frac{2|\sigma(\lambda)|\theta(\mathcal{A}_\mathcal{W})}{\zeta_K(2)}.\]	
	\end{thm}
	
	\begin{proof}
	Let $D\subset\bb{R}^2$ be a Jordan measurable set with $\vol(D)>0$. Let $T>0$ be given. By \Cref{lemInclExcl}, we have
	\begin{align}\theta(\widehat{\mathcal{A}_\mathcal{W}})&=\lim_{T\to\infty}\frac{\#(\widehat{\mathcal{A}_\mathcal{W}}\cap TD)}{\vol(TD)}\nonumber\\
	&=\lim_{T\to\infty}\sum_{\substack{F\subset C\\\#F<\infty}}(-1)^{\#F}\frac{\#\big(\big((\mathcal{A}_\mathcal{W})_*\cap \bigcap_{c\in F}c(\mathcal{A}_\mathcal{W})_*\big)\cap TD\big)}{\vol(TD)}\label{eqnAW2}
	\end{align}
	
	Note that for each finite subset $F$ of $C$, the corresponding term of the sum in \eqref{eqnAW2} tends to $\theta\Big(\mathcal{A}_\mathcal{W}\cap\bigcap_{c\in F} c\mathcal{A}_\mathcal{W}\Big)$. We begin by proving that the limit in \eqref{eqnAW2} can be calculated termwise.
	
	Let $\Delta>0$ be given. By \Cref{lemSizeOfPrimesA}, there are only finitely many $F\subset C$ with $N(\Pi_F)^2<\Delta$. We have
	\begin{align}
	& \left|\sum_{\substack{F\subset C,\#F<\infty\\N(\Pi_F)^2\geq \Delta}}(-1)^{\#F}\frac{\#\big(\big((\mathcal{A}_\mathcal{W})_*\cap \bigcap_{c\in F}c(\mathcal{A}_\mathcal{W})_*\big)\cap TD\big)}{\vol(TD)}\right|\nonumber\\
	&\quad \leq \sum_{\substack{F\subset C,\#F<\infty\\N(\Pi_F)^2\geq \Delta}}\frac{\#\{x\in\bb{Z}[\zeta]\Pi_F\mid \sigma(x)\in\mathcal{W}\}_*}{\vol(TD)}\leq L\sum_{\substack{F\subset C,\#F<\infty\\N(\Pi_F)^2\geq \Delta}}\frac{1}{N(\Pi_F)^2}\label{eqnEstimates},
	\end{align}
	where the first inequality follows from $-\frac{1}{\lambda}\mathcal{W}\subset \mc{W}$ together with \eqref{eqnAW} and the constant $L$, which is independent of $\Delta$, comes from \Cref{lemEstimate1}. Noting that the right-hand side of \eqref{eqnEstimates} tends to $0$ as $\Delta\to\infty$, we conclude that the limit in \eqref{eqnAW2} can be taken termwise. Hence \eqref{eqnAW} implies
	\[\theta(\widehat{\mathcal{A}_\mathcal{W}})=\sum_{\substack{F\subset C,\#F<\infty\\\lambda\in C}}(-1)^{\#F}\frac{\theta\left(\mathcal{A}_\mathcal{W}\right)}{\lambda^2N(\Pi_F)^2}+\sum_{\substack{F\subset C,\#F<\infty\\\lambda\notin C}}(-1)^{\#F}\frac{\theta\left(\mathcal{A}_\mathcal{W}\right)}{N(\Pi_F)^2},\]
	which is equal to $\left(1-\frac{1}{\lambda^2}\right)\frac{\theta(\mathcal{A}_\mathcal{W})}{\zeta_K(2)}=\frac{2|\sigma(\lambda)|\theta(\mathcal{A}_\mathcal{W})}{\zeta_K(2)}$ by \eqref{eqnDedekind}.
	\end{proof}
	
	From \eqref{eqnACPS} and \Cref{propDensity} it follows that $\theta(\mc{A}_\mc{W})=\frac{\vol(\mc{W})}{4}$. By using results from e.g.\ \cite[Chapter 4]{washington1997introduction}, one can show that $\zeta_K(2)=\frac{\pi^4}{48\sqrt{2}}$; thus the density of $\mathcal{A}_\mathcal{W}$ can be calculated explicitly. 
	
	For every $\mc{W}\in W_1$ with $\vol(\mc{W})>0$, \Cref{thmDensVisA} implies that the relative density of visible points in $\mathcal{A}_\mathcal{W}$ is $\frac{2|\sigma(\lambda)|}{\z_K(2)}=0.5773\ldots$, which is supported numerically by \cite[Table 2]{baake2014radial}. This result also agrees with the calculation in \cite{singppt1}, in the special case where $\mathcal{A}_\mathcal{W}$ is the Ammann--Beenker point set. We also remark that in this case $\theta(\widehat{\mathcal{A}_\mathcal{W}})=\frac{1}{\z_K(2)}$; note the resemblance with $\theta(\widehat{\Z^2})=\frac{1}{\z(2)}$. We provide some numerical support for this result in \Cref{table1} below.
	
	\begin{rem}	
	In \cite[pp. 34--38]{baake2000diffraction}, the set of visible points $\widehat{\Z^d}$ of $\Z^d$ is expressed as an \textit{adelic} cut-and-project set. More precisely, let $\pi$ be the projection from the $d$-adeles $\bb{A}_\bb{Q}^d$ onto $\R^d$ and $\pi_{\mathrm{int}}$ the projection onto the locally compact abelian group $\bb{A}_{\Q,f}^d$ of finite $d$-adeles. Let $\mc{W}=\prod_{p\in\bb{P}}(\Z_p^d\setminus p\Z_p^d)$, where $\bb{P}\subset \Z_+$ is the set of prime numbers. Let $\mathcal{L}$ be the image of the inclusion of $\Q^d$ in $\bb{A}_\bb{Q}^d$, a lattice in $\bb{A}_\bb{Q}^d$. Then
	\[\widehat{\Z^d}=\{\pi(x)\mid x\in \mathcal{L},\pi_{\mathrm{int}}(x)\in\mathcal{W}\}.\] Up to minor technical details, an application of the density formula \cite[Theorem 1]{schlottmann1998cut} for cut-and-project sets over locally compact abelian groups yields $\theta(\widehat{\Z^d)}=\frac{1}{\z(d)}$. In \cite{singppt1}, the density of visible points in the Ammann--Beenker point set was calculated via a similar adelic approach; it would be interesting to try this approach on other point sets.
	\end{rem}
	
	Next, recall from \eqref{eqnACPS} that $\mathcal{A}_\mathcal{W}=\mathcal{P}(\mathcal{W}\inv{A},\mathcal{L})A_1\subset\bb{R}^2$ for some invertible matrices $A,A_1$. As noted above, $\theta(\mc{A}_\mc{W})=\frac{\vol(\mc{W})}{4}$, thus in particular, if $\vol(\mathcal{W})=\vol(\mathcal{W}')$, then $\theta(\mathcal{A}_{\mathcal{W}})=\theta(\mathcal{A}_{\mathcal{W}'})$. This observation together with \Cref{thmDensVisA} implies the following corollary.
	
	\begin{cor}
	\label{corDensVisA1}
	If $\mathcal{W},\mathcal{W}'\in W_1$ satisfies $\vol(\mc{W})=\vol(\mc{W}')$ then $\theta(\widehat{\mathcal{A}_{\mathcal{W}}})=\theta(\widehat{\mathcal{A}_{\mathcal{W}'}})$.
	\end{cor}
	
	If $\mathcal{W}\in W_1$ and $x\in\bb{Z}[\zeta]$ are such that $\sigma(x)+\mathcal{W}\in W_1$, then, since
	\[x+\mc{A}_\mc{W}=\mc{A}_{\sigma(x)+\mc{W}}\]
	and $\vol(\mc{W})=\vol(\sigma(x)+\mc{W})$,
	\Cref{corDensVisA1} implies that $\theta(\widehat{x+\mc{A}_\mc{W}})=\theta(\widehat{\mc{A}_\mc{W}})$. Note that $y\in \widehat{x+\mc{A}_\mc{W}}$ if and only if $y$ is the $x$-translate of a point of $\mc{A}_\mc{W}$ \textit{visible from $-x$}. Thus, the density of the points of $\mc{A}_\mc{W}$ visible from $-x$ exists and is equal to $\theta(\widehat{\mc{A}_\mc{W}})$. If $\mc{W}+\e\in W_1$ for all sufficiently small $\epsilon\in\C$, the above holds for all $x\in\Z[\z]$ with $|\sigma(x)|$ sufficiently small. For instance, the octagon $\mc{W}$ defining the Ammann--Beenker point set has this property.
	
	The remainder of this section will be devoted to extending \Cref{thmDensVisA} to a more general result.
	Let $W_1'$ be the family of all Jordan measurable $\mathcal{W}\subset \bb{R}^2$ which are star-shaped with respect to the origin and contain a neighbourhood of the origin. Note that for each $\mathcal{W}\in W_1'$, there is some $r\geq 1$ with $-\mathcal{W}\subset r\mathcal{W}$ and the set of all primes $\pi\in \bb{P}$ with $|\sigma(\pi)|\leq r$ is finite.
	
	The following lemma provides a set of occlusion quotients for $\mathcal{A}_\mathcal{W}$ when $\mc{W}\in W_1'$.
	
	\begin{lem}
		\label{lemCextended}
		Fix $\mc{W}\in W_1'$ and $r\geq 1$ with $-\mc{W}\subset r\mc{W}$. Let $P=\{\pi_1,\ldots,\pi_n\}$ be the set of primes $\pi\in\bb{P}$ with $|\sigma(\pi)|\leq r$. Then, there $k_0,K,m_1,\ldots,m_n\in\Z$, $k_0\leq 0$, so that
		\begin{equation}
		\label{eqnLemCext}C:=(\bb{P}\setminus P)\cup\{\pi_1^{m_1},\ldots,\pi_n^{m_n}\}\cup\left\{\left.\lambda^k\prod_{i=1}^n\pi_i^{k_i}\right| k_0\leq k\leq K, -m_i< k_i< m_i\right\}\end{equation}
		is a set of occlusion quotients for $\mc{A}_\mc{W}$.
	\end{lem}

	\begin{proof}	
	Suppose $x\in \mathcal{A}_\mathcal{W}\setminus \widehat{\mathcal{A}_\mathcal{W}}$ and $\pi\mid x$ for some $\pi\in \bb{P}\setminus P$. Then $x/\pi\in \mc{A}_\mc{W}$ and $\pi\in C$. Next suppose that $x\in \mathcal{A}_\mathcal{W}\setminus \widehat{\mathcal{A}_\mathcal{W}}$ is divisible by primes in $P$ only. Note that for each $i$ there is an integer $m_i$ so that $x\in \widehat{\mc{A}_\mc{W}}$ if $\pi_i^{m_i}\mid x$. Thus, if $c:=\pi_i^{m}\mid x$ for some $m\geq m_i$, then $x/c\in \mc{A}_\mc{W}$ and $c\in C$. Suppose now, in addition to $x$ being divisible by primes in $P$ only, that the multiplicity of each $\pi_i$ in $x$ is less than $m_i$ and that $x/\lambda\notin \mc{A}_\mc{W}$. Find $c\in \Q(\sqrt{2})_{>1}$ so that $y:=x/c\in \widehat{\mathcal{A}_\mathcal{W}}$. Thus, $y$ is divisible by primes in $P$ only and the multiplicity of $\pi_i$ in $y$ is less than $m_i$.  Write $c=a/b$ for some relatively prime $a,b\in \Z[\sqrt{2}]$. From $c>1$ and $x/\lambda\notin \mc{A}_\mc{W}$ it follows that there are integers $k_0,K$ with $k_0\leq 0$ and $k_0\leq k\leq K$.
	\end{proof}

	We can now prove the following theorem, which gives $\theta(\widehat{\mathcal{A}_\mathcal{W}})$ for $\mc{W}\in W_1'$, and thus generalises \Cref{thmDensVisA}.
	
	\begin{thm}\label{thmDensVisA2}
		Fix $\mathcal{W}\in W_1'$ and $r\geq 1$ with $-\mc{W}\subset r\mc{W}$. Let $P=\{\pi_1,\ldots,\pi_n\}$ be the set of primes $\pi\in\bb{P}$ with $|\sigma(\pi)|\geq r$. Let 
		\[(\bb{P}\setminus P)\cup\{\pi_1^{m_1},\ldots,\pi_n^{m_n}\}\cup\left\{\left.\lambda^k\prod_{i=1}^n\pi_i^{k_i}\right| k_0\leq k\leq K, -m_i< k_i< m_i\right\}=:(\bb{P}\setminus P)\cup M\]
		be a set of occlusion quotients for $\mc{A}_\mc{W}$ as in \eqref{eqnLemCext}. Given a subset $M_0\subset M$, let $\Pi_{M_0}$ denote a least common multiple of its elements. Let $\mathcal{W}_{M_0}=\mathcal{W}\cap\bigcap_{c\in M_0}\sigma(c)\mathcal{W}$. Then
		\begin{equation}
		\label{eqnthmDensVisA2}
		\theta(\widehat{\mathcal{A}_\mathcal{W}})=\left(4\zeta_{\bb{Q}(\sqrt{2})}(2)\prod_{\pi\in P}\left(1-\frac{1}{N(\pi)^2}\right)\right)^{-1}\sum_{M_0\subset M}\frac{(-1)^{\#M_0}\vol(\mathcal{W}_{M_0})}{N(\Pi_{M_0})^2}.
		\end{equation}
	\end{thm}

	\begin{proof}
	Let $C=(\bb{P}\setminus P)\cup M$. By \Cref{lemInclExcl}, we have
	\begin{align*}
	\frac{\#(\widehat{\mathcal{A}_\mathcal{W}}\cap TD)}{\vol(TD)}&=\sum_{\substack{F\subset C\\ \#F<\infty}}(-1)^{\#F}\frac{\#((\mathcal{A}_\mathcal{W})_*\cap\bigcap_{c\in F}c(\mathcal{A}_\mathcal{W})_*\cap TD)}{\vol(TD)}\\
	&=\sum_{M_0\subset M}\sum_{\substack{F\subset C\\ \#F<\infty\\M\cap F=M_0}}(-1)^{\#F}\frac{\#((\mathcal{A}_\mathcal{W})_*\cap\bigcap_{c\in F}c(\mathcal{A}_\mathcal{W})_*\cap TD)}{\vol(TD)}.
	\end{align*}
	Note that for $M_0\subset M$ and a finite subset $F\subset C$ with $M\cap F=M_0$, we have \[\mathcal{A}_\mathcal{W}\cap\bigcap_{c\in F}c\mathcal{A}_\mathcal{W}=\left\{x\in\bb{Z}[\zeta]a_F\left| \sigma(x)\in \mathcal{W}\cap\bigcap_{c\in F}\sigma(c)\mathcal{W}\right.\right\},\]
	where $a_F$ is a least common multiple of the numerators of the elements of $F$. Note that $F\setminus M_0\subset \bb{P}\setminus P$ in this case. The fact that $\mathcal{W}\subset \sigma(\pi)\mathcal{W}$ for all $\pi\in\bb{P}\setminus P$ implies $\mathcal{W}\cap\bigcap_{c\in F}\sigma(c)\mathcal{W}=\mathcal{W}\cap\bigcap_{c\in M_0}\sigma(c)\mathcal{W}=\mathcal{W}_{M_0}$. 
	
	We can take $a_F=\Pi_{M_0}\Pi_{F\setminus M_0}$, where $\Pi_{M_0}$ is a least common multiple of the numerators of the elements of $M_0$ and $\Pi_{F\setminus M_0}$ is the product of the elements of $F\setminus M_0$. It follows from \eqref{eqnACPS} and \Cref{propDensity} that the density of $\{x\in\bb{Z}[\zeta]y\mid \sigma(x)\in\mathcal{W}\}$ for $y\in\bb{Z}[\sqrt{2}]$ exists and is equal to $\frac{\vol(\mathcal{W})}{4N(y)^2}$. Thus, by similar estimates as in \eqref{eqnEstimates}, we conclude that 
	\[\lim_{T\to\infty}\sum_{\substack{F\subset C\\ \#F<\infty\\M\cap F=M_0}}(-1)^{\#F}\frac{\#((\mathcal{A}_\mathcal{W})_*\cap\bigcap_{c\in F}(\mathcal{A}_\mathcal{W})_*\cap TD)}{\vol(TD)}=\sum_{\substack{F\subset \bb{P}\setminus P\\ \#F<\infty}}\frac{(-1)^{\#M_0+\#F}\vol(\mathcal{W}_{M_0})}{4N(\Pi_{M_0})^2N(\Pi_{F})^2},\]
	where $\Pi_F$ is the product of the elements of $F$. Since \[\sum_{\substack{F\subset \bb{P}\setminus P\\ \#F<\infty}}\frac{(-1)^{\#F}}{N(\Pi_{F})^2}=\prod_{\pi\in\bb{P}\setminus P}\left(1-\frac{1}{N(\pi)^2}\right)=\left(\zeta_{\bb{Q}(\sqrt{2})}(2)\prod_{\pi\in P}\left(1-\frac{1}{N(\pi)^2}\right)\right)^{-1},\]
	the theorem is proved.
	\end{proof}
	
	Let us now apply \Cref{thmDensVisA2} to a fixed $\mathcal{W}'\in W_1'\setminus W_1$. Let $\mathcal{W}$ be the open octagon such that $\mathcal{A}_\mathcal{W}$ is the Ammann--Beenker point set. One can show that $\mathcal{W}+\epsilon\in W_1$ for $\epsilon\in\bb{R}$ precisely when $|\epsilon|<\frac{\sqrt{2}-1}{2}$. Now take $\e=457-323\sqrt{2}>\frac{\sqrt{2}-1}{2}$ and let $\mathcal{W}'=\mathcal{W}+\e$. We then have $\mathcal{W}'\in W_1'\setminus W_1$, but it holds that $-\mathcal{W}'\subset \sigma(\pi)\mathcal{W}'$ for each $\pi\in\bb{P}\smpt{\sqrt{2}}$; hence we have $P=\{\sqrt{2}\}$. Note that $\mathcal{A}_{\mathcal{W}'}=\sigma(\epsilon)+\mathcal{A}_\mathcal{W}$, i.e.\ $\mathcal{A}_{\mathcal{W}'}$ is the translate of the Ammann--Beenker point set by the algebraic integer $457+323\sqrt{2}$. With notation as in \Cref{thmDensVisA2}, one can take $M=\{2,\lambda,\lambda^2,\sqrt{2},\sqrt{2}\lambda,\sqrt{2}\lambda^2,\frac{\lambda}{\sqrt{2}},\frac{\lambda^2}{\sqrt{2}}\}$. Note that by \eqref{eqnthmDensVisA2}, $\vol(\mathcal{W}_{M_0})$ must be calculated for each subset $M_0\subset M$. If $\mathcal{W}$ has a simple form, e.g.\ the shape of a polygon, so that $\mathcal{W}_{M_0}$ is an intersection of half-spaces, then this can be done numerically. In the present case $\mathcal{W}$ is a regular polygon and a numerical calculation of the sum in \eqref{eqnthmDensVisA2} gives $\theta(\widehat{\mc{A}_{\mc{W}'}})=\frac{c}{3\z_{\Q(\sqrt{2})}(2)}$, where $c=3.00057\ldots$ 
	(recall that $\theta(\widehat{\mc{A}_\mc{W}})=\frac{1}{\z_{\Q(\sqrt{2})}(2)}$ by \Cref{thmDensVisA}). By the remark following \Cref{corDensVisA1}, there are infinitely many $x\in\Z[\z]$ such that $\theta(\widehat{x+\mc{A}_\mc{W}})=\theta(\widehat{\mc{A}_\mc{W}})$, but the above example $\mathcal{A}_{\mathcal{W}'}=\sigma(\epsilon)+\mathcal{A}_\mathcal{W}$ shows that this does not hold for all $x\in\Z[\z]$.
	
	We end this discussion with \Cref{table1}, which contains numerical support to the above observation that the density of visible points of $\mc{A}_{\mc{W}'}$ is slightly greater than that of $\mc{A}_\mc{W}$. 
	
	\begin{table}[H]
		\centering	
		\begin{tabular}{lllll}
			\hline
			\rule{0pt}{3ex} $T$& $\widehat{N}_T$ & $\widehat{N}_T'$ & $\widehat{N}_T/\vol(B_T(0))$ \qquad& $\widehat{N}_T'/\vol(B_T(0))$\\
			\hline
			$1000$ \qquad\qquad& $2\,189\,104$ \qquad\qquad& $2\,189\,393$ \qquad\qquad& $0.696813\ldots$& $0.696905\ldots$\\
			$2500$ & $13\,683\,168$ & $13\,684\,733$ & $0.696877\ldots$& $0.696957\ldots$\\
			$3500$ & $26\,818\,928$ & $26\,823\,349$ & $0.696875\ldots$& $0.696990\ldots$\\\hline	
		\end{tabular}
		\caption{Numerical data for $\mc{A}_\mc{W}$ and $\mc{A}_{\mc{W}'}$, where $\mathcal{A}_\mc{W}$ is the Ammann--Beenker point set and $\mc{W}'=\mc{W}+457-323\sqrt{2}$; $\widehat{N}_T=\#(B_T(0)\cap \widehat{\mc{A}_{\mc{W}}})$ and $\widehat{N}_T'=\#(B_T(0)\cap \widehat{\mc{A}_{\mc{W}'}})$.}
		\label{table1}
	\end{table}

	Note that $\theta(\widehat{\mc{A}_\mc{W}})=\frac{1}{\z_{\Q(\sqrt{2})}(2)}=0.696877\ldots$ and $\theta(\widehat{\mc{A}_{\mc{W}'}})=\frac{c}{3\z_{\Q(\sqrt{2})}(2)}=0.697010\ldots$ with $c$ as above. The fourth and fifth columns of \Cref{table1} serve as approximations of $\theta(\widehat{\mc{A}_\mc{W}})$ and $\theta(\widehat{\mc{A}_{\mc{W}'}})$ respectively. We have done analogous computations for other values of $\epsilon\in \Z[\sqrt{2}]$ close to $\frac{\sqrt{2}-1}{2}$ with similar agreements of $\theta(\widehat{\mc{A}_{\mc{W}'}})$ to the corresponding numerical approximations as in \Cref{table1}.

	\subsection{$\theta(\widehat{\mathcal{T}_\mathcal{W}})$ for certain $\mathcal{W}$}
	\label{secdensTsets}
	
	Let $\z=e^{\frac{2\pi i}{5}}$, so that $K=\bb{Q}(\zeta)\cap \bb{R}=\bb{Q}(\tau)$, $\mathcal{O}_K=\bb{Z}[\tau]$ and $\bb{Z}[\zeta]=\mathcal{O}_K\oplus \mathcal{O}_K\zeta$. Let $\sigma$ be the automorphism of $K$ given by $\z\mapsto \z^2$. Note that $\mathcal{O}_K$ is a Euclidean domain with fundamental unit $\tau$. Let $W_2\subset \bb{C}$ denote the family of all Jordan measurable $\mathcal{W}\subset\bb{C}$ which are star-shaped with respect to the origin and satisfy $-\mathcal{W}\subset 2\tau\mathcal{W}$.
	
	The following results have counterparts in \Cref{lemSizeOfPrimesA} and \Cref{propVisCondA} with virtually identical proofs. Recall that $\bb{P}$ is the set of all primes $\pi\in \Z[\t]$ with $1<\pi<\t$.
	
	\begin{lem}
	\label{lemSizeOfPrimesT}
	For every $\pi\in\bb{P}$ we have $|\sigma(\pi)|\geq 2\tau$.
	\end{lem}
	
	\begin{prop}
	\label{propVisCondT}
	For $\mathcal{W}\in W_2$ we have
	\[\widehat{\mathcal{T}_\mathcal{W}}=\{x=x_1+x_2\zeta\in \mathcal{P}\mid \gcd(x_1,x_2)=1,\sigma(x/\tau)\notin \mathcal{W}\}.\]
	\end{prop}
	
	Let $C=\bb{P}\cup\{\tau\}$ and let $\mathcal{W}\in W_2$. Then, by \Cref{propVisCondT}, $C$ is a set of occlusion quotients for $\mc{T}_\mc{W}$. Proceeding in an analogous manner to the case for $\mathcal{A}$-sets we arrive at the following.

	\begin{thm}
	\label{thmDensVisT}
	For $\mathcal{W}\in W_2$ we have
	\[\theta(\widehat{\mathcal{T}_\mathcal{W}})=\sum_{\substack{F\subset C\\\#F<\infty}}(-1)^{\#F}\theta\Big(\mathcal{T}_\mathcal{W}\cap\bigcap_{c\in F} c\mathcal{T}_\mathcal{W}\Big)=\frac{|\sigma(\tau)|\theta(\mathcal{T}_\mathcal{W})}{\zeta_K(2)}\]
	where $\zeta_K(2)=\frac{2\sqrt{5}\pi^4}{375}$.
	\end{thm}

	Note that $\e+\mc{W}\in W_2$ for all regular decagons centered at the origin and $\e\in \C$ sufficiently small. Thus in particular, vertex sets $\mc{T}_\mc{W}$ from triangular tilings as in \cite{baake1990planar} are covered by \Cref{thmDensVisT}.

	We remark that it ought to be possible to prove an extension of \Cref{thmDensVisT} analogous to \Cref{thmDensVisA2}; however, we have not carried this out.
	
	\subsection{$\theta(\widehat{\mathcal{P}_\e})$ for $|\epsilon|<0.1$}
	\label{secdensPsets}
	
	Let $\z,K,\mc{O}_K,\t,\s$ be as in \Cref{secdensTsets}. Recall the definitions of $\mathcal{W}_{k,\epsilon}$, $\kappa$ and $\mathcal{P}_\epsilon$ from \Cref{secPsets}. Note that $\t=1+\z+\z^4$, hence $\kappa(\t)=3$. In \Cref{thmDenVisPenrose} below we give a formula for $\theta(\widehat{\mathcal{P}_\e})$ when $|\epsilon|<0.1$. 
	
	First, we verify that \Cref{thmdeBruijn} holds for each $\epsilon=\sum_{j=0}^4\gamma_j\zeta^{2j}$, where $\gamma\in(\bb{Q}\setminus\bb{Z})^5$ satisfies $\sum_{j=0}^4\gamma_j=0$. This result then allows us to explicitly provide $\epsilon\in\C$, $|\e|<0.1$, with the property that $\mc{P}_\e$ is the vertex set of a rhombic Penrose tiling.
	
	\begin{lem}
	\label{lemExOfRegularPenta}
	If $\gamma\in(\bb{Q}\setminus\bb{Z})^5$ satisfies $\sum_{j=0}^4\gamma_j=0$, then \[\e:=\sum_{j=0}^4\gamma_j\zeta^{2j}\notin\bigcup_{k=0}^5( \bb{R}\zeta^ki+(1-\zeta)).\]
	\end{lem}
	
	\begin{proof}
		Suppose, towards a contradiction, that  $\sum_{j=0}^4\gamma_j\zeta^{2j}=u\zeta^ki+\alpha$, for some $k\in\{0,\ldots,4\}$, $u\in\R$ and $\alpha\in (1-\z)$. Let $z=\z^{5-k}(-\alpha+\sum_{j=0}^4\gamma_j\zeta^{2j})\in \R i$. It follows that $z=\sum_{j=0}^4\gamma_j'\zeta^{j}\in\bb{R}i$ for some $\gamma'\in(\bb{Q}\setminus\bb{Z})^5$ with $\sum_{j=0}^4\gamma_j'\in 5\Z$. Using $z=-\overline{z}$, we find that
		\[0=2\g_0'+(\g_1'+\g_4')(\zeta+\zeta^4)+(\g_2'+\g_3')(\zeta^2+\zeta^3)=2\g_0'+(\g_1'+\g_4')(\tau-1)+(\g_2'+\g_3')(-\tau).\]
		Since $\t\in \R\setminus \Q$ we must have $\g_1'+\g_4'-\g_2'-\g_3'=0$ and also $2\g_0'-\g_1'-\g_4'=0$. Hence, $\gamma_1'+\gamma_4=\gamma_2'+\gamma_3'=2\gamma_0'$ and therefore $5\g_0=\sum_{j=0}^4\gamma_j'\in 5\Z$, which implies $\gamma_0'\in \Z$, contradiction.
	\end{proof}
	
	Henceforth we write $\pm D\subset E$ when $D\cup(-D)\subset E$. For all $\e\in \C$, with $|\epsilon|$ sufficiently small, we have for all $k_1,k_2\in\{1,2,3,4\}$ that $\mathcal{W}_{k_1,\epsilon}$ is star-shaped with respect to the origin and
	$\pm\frac{1}{2\tau}\mathcal{W}_{k_1,\epsilon}\subset \mathcal{W}_{k_2,\epsilon}$. This can be verified to hold when $|\epsilon|<0.1$.
	
	\begin{prop}
		\label{propVisCondP}
		For $\epsilon\in\C$ with $|\e|<0.1$ we have
		\[\widehat{\mathcal{P}_\epsilon}=\{x=x_1+x_2\zeta\in\mathcal{P}_\e\mid x_1,x_2\in\bb{Z}[\tau]: \gcd(x_1,x_2)=1, x/\tau\notin \mathcal{P}_\epsilon,x/\tau^2\notin \mathcal{P}_\epsilon\}.\]
	\end{prop}

	\begin{proof}
	First necessity of the visibility conditions are proved. Take $x=x_1+x_2\zeta\in \mathcal{P}_\e$. If $x/\tau\in \mathcal{P}_\epsilon$ or $x/\tau^2\in \mathcal{P}_\epsilon$, then $x$ is invisible. If $x_1,x_2$ are not relatively prime, then there is some prime $\pi \in \bb{P}$ such that $\pi \mid x_1,x_2$. We must have $\kappa(\pi)\neq 0$, hence $1-\zeta\nmid \pi$ in $\Z[\z]$. The only prime in $\bb{P}$ divisible by $1-\zeta$ is $3-\t$, which is the prime in $\bb{P}$ dividing $5$. Thus, $\pi\in\bb{P}\smpt{3-\t}$. By \Cref{lemSizeOfPrimesT}, and the fact that $\pm\frac{1}{2\tau}\mathcal{W}_{k_1,\epsilon}\subset \mathcal{W}_{k_2,\epsilon}$ for all $k_1,k_2$, we conclude that $x/\pi\in \mathcal{P}_\e$, and thus $x$ is invisible.
	
	To prove sufficiency, take $x\in \mathcal{P}_{\epsilon}\setminus \widehat{\mathcal{P}_{\epsilon}}$. Then, there is some $\alpha\in \bb{R}_{>1}$ such that $x/\alpha\in \mathcal{P}_{\epsilon}\subset \bb{Z}[\zeta]$. Since $\mathcal{P}_{\epsilon}$ is locally finite, we may assume that $y:=x/\alpha\in \widehat{\mathcal{P}_{\epsilon}}$. By the necessary conditions proved above, if we write $y=y_1+y_2\zeta$ with $y_1,y_2\in\Z[\t]$, then $y_1,y_2$ must be relatively prime. Hence $\alpha\in \bb{Q}(\zeta)\cap \bb{R}=\bb{Q}(\tau)$. Write $\alpha=a_1/a_2$ for some relatively prime $a_i\in\bb{Z}[\tau]$. Since $y_1,y_2$ are relatively prime, $a_2$ has to be a unit, i.e.\ $\alpha\in \bb{Z}[\tau]$.
	
	If $|\sigma(\alpha)|>1$, then $x_1,x_2$ are not relatively prime. Otherwise, $\alpha=\tau^k$ for some $k>1$. If $k\geq 4$ then $\sigma(x/\tau^k)=(-1)^k\tau^k\sigma(x)\in \mathcal{W}_{k_1,\epsilon}$ for $k_1=\kappa(x/\tau^k)$. Also, we have \[\sigma(x/\tau)=-\tau\sigma(x)\in \frac{(-1)^{k+1}}{\tau^{k-1}}\mathcal{W}_{k_1,\epsilon}\subset \frac{(-1)^{k+1}}{\tau^{3}}\mathcal{W}_{k_1,\epsilon}\subset \frac{(-1)^{k+1}}{2\tau}\mathcal{W}_{k_1,\epsilon}\subset \mathcal{W}_{k_2,\epsilon}\]
	for all $k_2$, hence $x/\tau\in \mathcal{P}_\epsilon$.
	
	Suppose now $k=3$. For each of the four possible values of $\kappa(x)$ we verify that $x/\tau\in \mathcal{P}_\epsilon$. The case $\kappa(x)=1$ is showed, the other cases can be treated similarly. In this case, $\kappa(x/\tau)=2$ and $\kappa(x/\tau^3)=3$ so $\sigma(x/\tau^3)=-\tau^3\sigma(x)\in \mathcal{W}_{3,\epsilon}$. Hence, $-\tau\sigma(x)=\sigma(x/\tau)\in \frac{\mathcal{W}_{3,\epsilon}}{\tau^2}\subset\mathcal{W}_{2,\epsilon}$, that is $x/\tau\in \mathcal{P}_\epsilon$. The inclusion $\frac{\mathcal{W}_{3,\epsilon}}{\tau^2}\subset\mathcal{W}_{2,\epsilon}$ is guaranteed by $|\e|<0.1$.
	\end{proof}

	As a by-product of the proof of \Cref{propVisCondP}, we find that \[C:=(\bb{P}\smpt{3-\t})\cup\{\t,\t^2\}\]
	is a set of occlusion quotients for $\mc{P}_\e$ if $|\e|<0.1$.
	
	\begin{lem}
		\label{lemDensityP1}
		For each $k\in\{1,\ldots,4\}$, $y\in \Z[\t]\setminus (3-\t)$ and Jordan measurable $\mathcal{W}\subset \bb{R}^2$ we have, with $\mc{P}=\{x\in\bb{Z}[\zeta]: \kappa(x)=k, y\mid x,\sigma(x)\in\mathcal{W}\}$, that
		$\theta(\mc{P})=\frac{4\vol(\mathcal{W})}{25(2\tau-1)N(y)^2}$.
		
		In particular, 
		\[\theta(\mc{P}_\e)=\frac{8(1+\t^2)\vol(\mc{W}_1)}{25(2\t-1)}\]
		for any $\e\in\C$, where $\mc{W}_1$ is the open regular pentagon with vertices $1,\z,\z^2,\z^3,\z^4$.
	\end{lem}
	
	\begin{proof}
		By observing that $\k(x)=k$ if and only if $x\in k+(1-\z)$ we see that $\mc{P}$ can be identified with the translate of a set of the form \eqref{eqnTCPS}, whose density can be calculated by \Cref{propDensity} and \Cref{lemDensitiesGeneral}, and the first claim follows.
		
		The formula for $\theta(\mc{P}_\e)$ follows from the definition of $\mc{P}_\e$ and the first claim.
	\end{proof}
	
	\begin{lem}
		\label{lemPenroseBig} Fix $\epsilon\in\C$ with $|\e|<0.1$.
		For a finite subset $F\subset \bb{P}\smpt{3-\tau}$, let $\Pi_F$ denote the product of the elements of $F$. 
		
		\begin{itemize}
			\item[\emph{(i)}] Let $\mathcal{P}_\e(F,\t)=\mathcal{P}_\epsilon\cap \tau\mathcal{P}_\epsilon\cap\bigcap_{\pi\in F}\pi\mathcal{P}_\epsilon$. Then \[\mathcal{P}_\e(F,\t)=\bigcup_{k=1}^4\{x\in\bb{Z}[\zeta]: \kappa(x)=k,\Pi_F\mid x,\sigma(x)\in \mc{W}'_{k,\e}\},\]
			where $\mc{W}'_{1,\e}=\mc{W}_{1,\e}\cap \mc{W}_{1,-\e/\t}$, $\mc{W}'_{2,\e}=\inv{\t}\mc{W}_{1,-\e}$, $\mc{W}'_{3,\e}=-\inv{\t}\mc{W}_{1,\e}$ and $\mc{W}'_{4,\e}=(-\mc{W}_{1,-\e})\cap (-\mc{W}_{1,\e/\t})$.
			
			\item[\emph{(ii)}] Let $\mathcal{P}_\e(F,\t^2)=\mathcal{P}_\epsilon\cap \tau^2\mathcal{P}_\epsilon\cap\bigcap_{\pi\in F}\pi\mathcal{P}_\epsilon$. Then \[\mathcal{P}_\e(F,\t^2)=\bigcup_{k=1}^4\{x\in\bb{Z}[\zeta]: \kappa(x)=k,\Pi_F\mid x,\sigma(x)\in \mc{W}'_{k,\e}\},\]
			where $\mc{W}'_{1,\e}=-\t^{-2}\mc{W}_{1,-\e}$, $\mc{W}'_{2,\e}=\inv{\t}\mc{W}_{1,\e/\t}$, $\mc{W}'_{3,\e}=-\inv{\t}\mc{W}_{1,-\e/\t}$ and $\mc{W}'_{4,\e}=\t^{-2}\mc{W}_{1,\epsilon}$.
			
			\item[\emph{(iii)}] Let $\mathcal{P}_\e(F,\t,\t^2)=\mathcal{P}_\epsilon\cap\tau\mathcal{P}_\epsilon \cap \tau^2\mathcal{P}_\epsilon\cap \bigcap_{\pi\in F}\pi\mathcal{P}_\epsilon$. Then
			\[\mathcal{P}_\e(F,\t,\t^2)=\bigcup_{k=1}^4\{x\in\bb{Z}[\zeta]: \kappa(x)=k,\Pi_F\mid x,\sigma(x)\in \mc{W}'_{k,\e}\},\]
			where $\mc{W}'_{1,\e}=-\t^{-2}\mc{W}_{1,-\e}$, $\mc{W}'_{2,\e}=\inv{\t}(\mc{W}_{1,-\e}\cap \mc{W}_{1,\e/\t})$, $\mc{W}'_{3,\e}=-\inv{\t}(\mc{W}_{1,\e}\cap \mc{W}_{1,-\e/\t})$ and $\mc{W}'_{4,\e}=\t^{-2}\mc{W}_{1,\epsilon}$.
			
			\item[\emph{(iv)}] The densities of $\mathcal{P}_\e(F,\t)$, $\mathcal{P}_\e(F,\t^2)$ and $\mathcal{P}_\e(F,\t,\t^2)$ exist and are equal to
			\[\frac{4\sum_{k=1}^4\vol(\mathcal{W}'_{k,\e})}{25(2\tau-1)N(\Pi_F)^2},\]
			with the appropriate $\mathcal{W}'_{k,\e}$ defined in \emph{(i)}--\emph{(iii)}.
		\end{itemize}
	\end{lem}
	
	\begin{proof}
	\begin{enumerate}[(i)]
		\item This equality is proved by treating each of the cases $\k(x)=k$ separately. These cases are similar, hence we will only discuss the case $k=1$ here.
		
		Take $x$ in the left hand side of the equality with $\k(x)=1$. Then, $\k(x/\t)=2$ so $\sigma(x/\t)=-\t\sigma(x)\in \mc{W}_{2,\e}$ which implies that $\sigma(x)\in \mc{W}'_{1,\e}$. Since $\Pi_F\mid x$, it follows that $x$ is an element of the right hand-side. For the reverse inclusion, note that  $\pm\frac{1}{\tau^2}\mathcal{W}'_{k_1,0}\subset \mathcal{W}'_{k_2,0}$ for all $k_1,k_2$. Thus also $\pm\frac{1}{2\tau}\mathcal{W}'_{k_1,\e}\subset \mathcal{W}'_{k_2,\e}$ holds for $\e$ sufficiently small. This can be verified to hold for $|\e|<0.1$ whence the conclusion follows by \Cref{lemSizeOfPrimesT}.
		
		\item As in (i), we discuss the case $k=1$ only. Again it is straightforward to verify that $x$ in the left hand side with $\k(x)=1$ belongs to the right-hand side. For the reverse inclusion note again that $\pm\frac{1}{\tau^2}\mathcal{W}'_{k_1,\e}\subset \mathcal{W}'_{k_2,\e}$ for all $k_1,k_2$ for $\e=0$ whence $\pm\frac{1}{2\tau}\mathcal{W}'_{k_1,\e}\subset \mathcal{W}'_{k_2,\e}$ holds for sufficiently small $\e$ as well, in particular for $|\e|<0.1$.
		
		\item As in (i), we discuss the case $k=1$ only. Take $x$ in the left hand side with $\k(x)=1$. Then $\kappa(x/\tau)=2$ and $\kappa(x/\tau^2)=4$, so $\sigma(x/\tau)=-\tau\sigma(x)\in \mathcal{W}_{2,\epsilon}$ and $\sigma(x/\tau^2)=\tau^2\sigma(x)\in \mathcal{W}_{4,\epsilon}$. Thus, 
		\[\sigma(x)\in \mathcal{W}_{1,\epsilon}\cap-\inv{\t}\mathcal{W}_{2,\epsilon} \cap\t^{-2}\mathcal{W}_{4,\epsilon}=\t^{-2}\mathcal{W}_{4,\epsilon},\]
		since $|\e|<0.1$. We have that $\t^{-2}\mathcal{W}_{4,\epsilon}=\mc{W}'_{1,\e}$ so $x$ belongs to the right-hand side.
		
		For the reverse inclusion, note again that $\pm\frac{1}{\tau^2}\mathcal{W}'_{k_1,\e}\subset \mathcal{W}'_{k_2,\e}$ for all $k_1,k_2$ for $\e=0$, whence $\pm\frac{1}{2\tau}\mathcal{W}'_{k_1,\e}\subset \mathcal{W}'_{k_2,\e}$ holds for sufficiently small $\e$ as well, in particular for $|\e|<0.1$.
		
		\item This is a consequence of \Cref{lemDensityP1}.
	\end{enumerate}
	\end{proof}

	We can now prove the following theorem, which gives the density of visible points of $\mc{P}_\e$ for $|\e|<0.1$.
	
	\begin{thm}
		\label{thmDenVisPenrose} For $\e\in\C$ with $|\e|<0.1$ we have
		\[\theta(\widehat{\mc{P}_\e})=\frac{(3+\t)\vol(\mc{W}_1)-\vol(\mc{W}_1\cap(\mathcal{W}_1+\t\e))}{3(\t+2)\z_K(2)}.\]
	\end{thm}

	\begin{proof}
	Observe that $0\notin \mathcal{P}_\epsilon$, since $\kappa(0)=0$. Hence $(\mathcal{P}_\epsilon)_*=\mathcal{P}_\epsilon$. An application of \Cref{lemInclExcl}, with $C=(\bb{P}\smpt{3-\t})\cup\{\t,\t^2\}$, yields for any Jordan measurable $D\subset \R^2$ with $\vol(D)>0$
	\begin{align}
		\frac{\#(\widehat{\mc{P}_\e}\cap TD)}{\vol(TD)} &= \sum_{\substack{F\subset \bb{P}\setminus\{3-\t\}\\\#F<\infty}}(-1)^{\#F}\frac{\#(\mathcal{P}_\e\cap \bigcap_{\pi\in F}\pi\mathcal{P}_\e\cap TD)}{\vol(TD)}\label{eqnEstimatesPenrose1-1}\\
		&\qquad-\sum_{\substack{F\subset \bb{P}\setminus\{3-\t\}\\\#F<\infty}}(-1)^{\#F}\frac{\#(\mathcal{P}_\e\cap\t\mathcal{P}_\e\cap \bigcap_{\pi\in F}\pi\mathcal{P}_\e\cap TD)}{\vol(TD)}\label{eqnEstimatesPenrose1-2}\\
		&\qquad-\sum_{\substack{F\subset \bb{P}\setminus\{3-\t\}\\\#F<\infty}}(-1)^{\#F}\frac{\#(\mathcal{P}_\e\cap\t^2\mathcal{P}_\e\cap \bigcap_{\pi\in F}\pi\mathcal{P}_\e\cap TD)}{\vol(TD)}\label{eqnEstimatesPenrose1-3}\\
		&\qquad+\sum_{\substack{F\subset \bb{P}\setminus\{3-\t\}\\\#F<\infty}}(-1)^{\#F}\frac{\#(\mathcal{P}_\e\cap\t\mathcal{P}_\e\cap\t^2\mathcal{P}_\e\cap \bigcap_{\pi\in F}\pi\mathcal{P}_\e\cap TD)}{\vol(TD)}.\label{eqnEstimatesPenrose1-4}
	\end{align}
	
	To obtain $\theta(\widehat{\mc{P}_\e})$, we let $T\to\infty$ in the above. In the right-hand side we have to switch order of limit and summation. We show that this is possible for the fourth term \eqref{eqnEstimatesPenrose1-4}; one can proceed analogously with the other terms. To this end, let $\Delta>0$ be given. The terms of \eqref{eqnEstimatesPenrose1-4} converge to \[(-1)^{\#F}\theta(\mc{P}_\e(F,\t,\t^2))=(-1)^{\#F}\frac{\theta(\mc{P}_\e\cap\t\mc{P}_\e\cap \t^2\mc{P}_\e)}{N(\Pi_F)^2}\]
	as $T\to\infty$, where $\Pi_F$ denotes the product of the elements of $F$. By \Cref{lemSizeOfPrimesT} there are only finitely many finite subsets $F\subset C$ with $N(\Pi_F)^2<\Delta$. Now
	\begin{align}
		&\left|\lim_{T\to\infty}\sum_{\substack{F\subset \bb{P}\setminus\{3-\t\}\\\#F<\infty}}(-1)^{\#F}\frac{\#(\mathcal{P}_\e(F,\t,\t^2)\cap TD)}{\vol(TD)}-\sum_{\substack{F\subset \bb{P}\setminus\{3-\t\}\\\#F<\infty}}(-1)^{\#F}\theta(\mc{P}_\e(F,\t,\t^2))\right|\nonumber\\
		& \leq \lim_{T\to\infty}\sum_{\substack{F\subset \bb{P}\setminus\{3-\t\}\\\#F<\infty\\N(\Pi_F)^2\geq \Delta}}\frac{\#(\mathcal{P}_\e(F,\t,\t^2)\cap TD)}{\vol(TD)}+\theta(\mc{P}_\e\cap\t\mc{P}_\e\cap \t^2\mc{P}_\e)\sum_{\substack{F\subset \bb{P}\setminus\{3-\t\}\\\#F<\infty\\N(\Pi_F)^2\geq \Delta}}\frac{1}{N(\Pi_F)^2}\label{eqnEstimatesPenrose2}
	\end{align}
	By \Cref{lemPenroseBig} (iii) we have
	\[\mathcal{P}_\e(F,\t,\t^2)\subset \{x\in \Z[\z]:\Pi_F\mid x,\sigma(x)\in \mathcal{W}\},\]
	with e.g.\ $\mc{W}=\bigcup_{k=1}^4\mc{W}'_{k,\e}$ and thus by \Cref{lemEstimate1}, there is a constant $L$ independent of $\Delta$ such that \eqref{eqnEstimatesPenrose2} is bounded by
	\[\sum_{\substack{F\subset \bb{P}\setminus\{3-\t\}\\\#F<\infty\\N(\Pi_F)^2\geq \Delta}}\frac{L}{N(\Pi_F)^2}+\theta(\mc{P}_\e\cap\t\mc{P}_\e\cap \t^2\mc{P}_\e)\sum_{\substack{F\subset \bb{P}\setminus\{3-\t\}\\\#F<\infty\\N(\Pi_F)^2\geq \Delta}}\frac{1}{N(\Pi_F)^2}\]
	which goes to $0$ as $\Delta\to\infty$ since \[\sum_{\substack{F\subset \bb{P}\setminus\{3-\t\}\\\#F<\infty\\N(\Pi_F)^2\geq \Delta}}\frac{1}{N(\Pi_F)^2}
	\leq \sum_{I\in \mathfrak{I}, N(I)^2\geq \Delta}\frac{|\mu(I)|}{N(I)^2}.\] Similar treatment of the sums \eqref{eqnEstimatesPenrose1-1}--\eqref{eqnEstimatesPenrose1-3} yields
	\[\theta(\widehat{\mc{P}_\e})=(\theta(\mc{P}_\e)-\theta(\mc{P}_\e\cap\t\mc{P}_\e)-\theta(\mc{P}_\e\cap\t^2\mc{P}_\e)+\theta(\mc{P}_\e\cap\t\mc{P}_\e \cap\t^2\mc{P}_\e))\sum_{\substack{F\subset \bb{P}\setminus\{3-\t\}\\\#F<\infty}}\frac{(-1)^{\#F}}{N(\Pi_F)^2}.\]
	From the second part of \Cref{lemDensityP1} and \Cref{lemPenroseBig} (iv) it follows that \[\frac{\theta(\mc{P}_\e)}{\t^4}=\theta(\mc{P}_\e\cap\t^2\mc{P}_\e) \qquad\text{and}\qquad \frac{\theta(\mc{P}_\e\cap\t\mc{P}_\e)}{\t^2}=\theta(\mc{P}_\e\cap\t\mc{P}_\e \cap\t^2\mc{P}_\e).\] Therefore
	\begin{align*}
		\theta(\widehat{\mc{P}_\e})&=((1-1/\t^4)\theta(\mc{P}_\e)-(1-1/\t^2)\theta(\mc{P}_\e\cap\t\mc{P}_\e))\sum_{\substack{F\subset \bb{P}\setminus\{3-\t\}\\\#F<\infty}}\frac{(-1)^{\#F}}{N(\Pi_F)^2}\\
		&=\frac{8(3+\t)\vol(\mc{W}_1)-8\vol(\mc{W}_1\cap(\mathcal{W}_1+\t\e))}{25(\t+2)}\sum_{\substack{F\subset \bb{P}\setminus\{3-\t\}\\\#F<\infty}}\frac{(-1)^{\#F}}{N(\Pi_F)^2},
	\end{align*}
	where 
	\begin{align*}
		\sum_{\substack{F\subset \bb{P}\setminus\{3-\t\}\\\#F<\infty}}\frac{(-1)^{\#F}}{N(\Pi_F)^2}&=\prod_{\pi\in \bb{P}\smpt{3-\t}}\left(1-\frac{1}{N(\pi)^2}\right)\\
		&=\left(1-\frac{1}{N(3-\t)^2}\right)^{-1}\prod_{P\in\mathfrak{P}}\left(1-\frac{1}{N(P)^2}\right)=\frac{25}{24\z_K(2)},
	\end{align*}
	and the proof is complete.
	\end{proof}

	We conclude this section by presenting some numerical support for \Cref{thmDenVisPenrose}, in the case of a particular $\e=\e_0$. Let $\gamma=\frac{1}{101}(2,1,-2-2,1)$ and set $\e_0=\sum_{j=0}^4\gamma_j\z^{2j}=-0.0084\ldots$. Note that $|\e_0|<0.1$ and that $\mc{P}_{\e_0}$ is the vertex set of a rhombic Penrose tiling by \Cref{lemExOfRegularPenta}. A numerical calculation of $\theta(\widehat{\mc{P}_{\e_0}})$ using \Cref{thmDenVisPenrose} yields 
	\[\theta(\widehat{\mc{P}_{\e_0}})=0.684307\ldots,\]
	compare with the third column of \Cref{table2} below. 
	
	\begin{table}[H]
		\centering	
		\begin{tabular}{lll}
			\hline
			\rule{0pt}{3ex} $T$\qquad\qquad& $\widehat{N}_T$\qquad\qquad &  $\widehat{N}_T/\vol(B_T(0))$ \qquad\qquad\\\hline
			$1500$ & $4\,835\,583$ & $0.684095\ldots$\\
			$1800$ & $6\,964\,297$ & $0.684198\ldots$\\
			$2000$ & $8\,599\,221$ & $0.684304\ldots$\\\hline
		\end{tabular}
		\caption{Numerical data for $\mc{P}_{\e_0}$, where $\e_0=\sum_{j=0}^4\gamma_j\z^{2j}$, $\gamma=\frac{1}{101}(2,1,-2-2,1)$ and $\widehat{N}_T=\#(B_T(0)\cap \widehat{\mc{P}_{\e_0}})$.}
		\label{table2}
	\end{table}
	
	\section{Calculation of $m_{\widehat{\mathcal{P}}}$}
	\label{secCalcMPhat}

	Given a unital ring $R$, let $\mathrm{ASL}(n,R)$ be $\SL{n}{R}\times R^n$ endowed with the group operation
	\[(A_1,v_1)(A_2,v_2)=(A_1A_2,v_1A_2+v_2).\] Let $n=d+m$ be given. Let $G=\ASL{n}{\R}$ and $\Gamma=\ASL{n}{\Z}$. Note that $G$ acts on $\R^n$ by $rg=rA+v$ for $r\in\R^n$ and $g=(A,v)\in G$. Hence $\mathcal{L}=\delta^{1/n}(\bb{Z}^ng)$ is an affine lattice for every $\d>0$, $g\in G$, and every affine lattice in $\R^n$ can be represented in this way. Let $\varphi_g:\ASL{d}{\R}\longrightarrow\ASL{n}{\R}$ be the map $(A,v)\mapsto g\left(\begin{pmatrix}
	A & 0\\
	0 & I_m
	\end{pmatrix},(v,0)\right)\inv{g}$. By the results of Ratner \cite{ratner1991raghunathanAnnals,ratner1991raghunathan} there exists a unique, closed, connected subgroup $H_g\subset G$ such that $\Gamma\cap H_g\subset H_g$ is a lattice, $\varphi_g(\SL{d}{\R})\subset H_g$ and the closure of $\Gamma\backslash\Gamma \varphi_g(\SL{d}{\R})$ in $\Gamma\backslash G$ is $\Gamma\backslash\Gamma H_g$. These results also imply the existence of a unique, closed, connected subgroup $\widetilde{H}_g\subset G$ such that $\Gamma\cap \widetilde{H}_g\subset \widetilde{H}_g$ is a lattice, $\varphi_g(\ASL{d}{\R})\subset \widetilde{H}_g$ and the closure of $\Gamma\backslash\Gamma \varphi_g(\ASL{d}{\R})$ in $\Gamma\backslash G$ is $\Gamma\backslash\Gamma \widetilde{H}_g$.
	
	Let $X$ be the homogeneous space $X=(\Gamma\cap H_g)\backslash H_g$.  Note that $X$ can be identified with $\Gamma\backslash\Gamma H_g$; let $\mu$ be the $H_g$-invariant probability measure on either of these spaces. Fix a bounded set $\mathcal{W}\subset \overline{\pi_{\mathrm{int}}(\mathcal{L})}\subset \R^m$ and define for $x=\Gamma h\in X$
	\[\mc{P}^x=\mc{P}(\mc{W},\delta^{1/n}(\Z^nhg))\subset\R^d.\]
	By taking a random $x\in X$ with respect to $\mu$, a point process $x\mapsto \mc{P}^x$ on $\R^d$ consisting of cut-and-project sets is obtained. This process is $\SL{d}{\R}$-invariant since $\varphi_g(\SL{d}{\R})\subset H_g$. The process $x\mapsto \mc{P}^x$ and the space $\{\mc{P}^x\mid x\in X\}$ were introduced in \cite{marklof2014free}. 
	
	Let now $\mathcal{P}=\mc{P}(\mc{W},\mc{L})\subset \R^2$ be a regular cut-and-project and let $F:\R\longrightarrow [0,1]$ be the limiting distribution of normalised gaps in $\mc{P}$ as defined in the introduction. In \cite{marklof2014visibility} it is shown that
	\[F(s)=-\frac{d}{ds}\mu(\{x\in X\mid \#(\widehat{\mc{P}^x}\cap\mathfrak{C}(\inv{\k_\mc{P}}s))=0\})\]
	for each $s>0$ where
	\[\mathfrak{C}(s)=\left\{(x_1,x_2)\in\R^2: 0<x_1<1, |x_2|<s/\theta(\mc{P})\right\}\]
	and $\k_{\mc{P}}=\frac{\theta(\widehat{\mc{P}})}{\theta(\mc{P})}$. In \cite[Section 12]{marklof2014visibility}, $m_{\widehat{\mc{P}}}$ is defined as
	\begin{equation}
		\label{defnMPhat2}
		m_{\widehat{\mc{P}}}=\sup\{s\geq 0\mid \#(\widehat{\mc{P}^x}\cap \mathfrak{C}(\inv{\k_\mc{P}}s))\leq 1\text{ for }\mu \text{-a.e. }x\in X\}
	\end{equation}
	and it is shown that $m_{\widehat{\mathcal{P}}}=\sup\{\sigma\geq 0\mid F(s)=1~\mathrm{for~ all}~s\in[0,\sigma]\}$. Thus the definition of $m_{\widehat{\mc{P}}}$ given in \eqref{defnMPhat2} is equivalent with the definition given in \eqref{defnMPhat1}. 
	The $\SL{2}{\R}$-invariance of the process $x\mapsto \mc{P}^x$ implies that $m_{\widehat{\mc{P}A}}=m_{\widehat{\mc{P}}}$ for all $A\in\SL{2}{\R}$. This invariance also implies that the value of $m_{\widehat{\mc{P}}}$ remains unaffected if $\mathfrak{C}(\inv{\k_\mc{P}}s)$ in \eqref{defnMPhat2} is replaced by any other triangle with one vertex at the origin and with equal area. We now claim that $m_{\widehat{\mc{P}}}=m_{\widehat{c\mc{P}}}$ for each $c>0$. For $x\in X$ we have $(c\mc{P})^x=c\mc{P}^x$. Hence, by \eqref{defnMPhat2} we have 
	\begin{equation*}
	m_{\widehat{c\mc{P}}}=\sup\{s\geq 0\mid \#(\widehat{c\mc{P}^x}\cap \mathfrak{C}(\inv{\k_{c\mc{P}}}s))\leq 1\text{ for }\mu \text{-a.e. }x\in X\}.
	\end{equation*} In view of \Cref{lemDensitiesGeneral} we have $\theta(\widehat{c\mc{P}})=c^{-2}\theta(\widehat{\mc{P}})$, which implies that the triangles $\inv{c}\mathfrak{C}(\inv{\k_{c\mc{P}}}s)$ and $\mathfrak{C}(\inv{\k_{\mc{P}}}s)$ have the same area. It follows that $m_{\widehat{c\mc{P}}}=m_{\widehat{\mc{P}}}$ and therefore also that $m_{\widehat{\mc{P}}}$ is invariant under $\mc{P}\mapsto \mc{P}A$ for $A\in\GL{2}{\R}$ with positive determinant.
	
	We now prove that the minimal gap of a regular cut-and-project set remains unchanged when replacing the window defining the cut-and-project set by its closure.
	
	\begin{lem}
		\label{lemMPhatClosure} Let $\mathcal{P}=\mc{P}(\mc{W},\mc{L})\subset \R^2$ be a regular cut-and-project set and let $\mathcal{P}_1=\mc{P}(\overline{\mc{W}},\mc{L})$, where $\overline{\mc{W}}$ is the closure of $\mc{W}$ in $\overline{\pi_{\mathrm{int}}(\mc{L})}$. Then $m_{\widehat{\mathcal{P}}}=m_{\widehat{\mathcal{P}_1}}$.
	\end{lem}

	\begin{proof}
		Suppose $\mc{L}=\delta^{1/n}(\Z^ng)$ for some $g\in G$ and $\delta>0$. Since $\mc{P}(\mc{W},c\mc{L})=c\mc{P}(\inv{c}\mc{W},\mc{L})$ for $c>0$ and $m_{\widehat{c\mc{P}}}=m_{\widehat{\mc{P}}}$, we may assume that $\delta=1$. Let $X=\Gamma\bs \Gamma H_g$. Since $\mc{W}\subset \overline{\mc{W}}$, we have $\mc{P}^x\subset \mc{P}_1^x$ for all $x\in X$ and hence $m_{\widehat{\mathcal{P}}}\geq m_{\widehat{\mathcal{P}_1}}$ by \eqref{defnMPhat2}. Next, it is shown that $m_{\widehat{\mathcal{P}}}\leq m_{\widehat{\mathcal{P}_1}}$. To this end, take $s_0<m_{\widehat{\mathcal{P}}}$. Then
		$X':=\{x\in X\mid\#(\widehat{\mc{P}^x}\cap \mathfrak{C}(\inv{\k_\mc{P}}s_0))\leq 1\}$ satisfies $\mu(X')=1$. Let $X''\subset X'$ be the set $\{x\in X\mid\#(\widehat{\mc{P}_1^x}\cap \mathfrak{C}(\inv{\k_{\mc{P}_1}}s_0))\leq 1\}$. We will show that $\mu(X'')=1$ as well.
		
		By assumption, $\partial\mathcal{W}$ has measure $0$ with respect to Haar measure on $\overline{\pi_{\mathrm{int}}(\mc{L})}$. Hence, by applying \cite[Theorem 5.1]{marklof2014free} with $f=1_{\R^d\times \partial\mc{W}}$, we conclude that $ \Z^nhg\cap (\R^d\times \partial\mc{W})=\emptyset$ for $\mu$-almost every $\Gamma h=x\in X$. Now, the remark following \Cref{lemDensitiesGeneral2} gives $\k_{\mc{P}}=\k_{\mc{P}_1}$. Take $x\in X'\setminus X''$ and write $x=\Gamma h$ for some $h\in H_g$. Then, $\#(\widehat{\mc{P}_1^x}\cap \mathfrak{C}(\inv{\k_{\mc{P}}}s_0))\geq 2$ and $\#(\widehat{\mc{P}^x}\cap \mathfrak{C}(\inv{\k_{\mc{P}}}s_0))\leq 1$ hold. It follows that there is $y\in \Z^nhg$ with $\pi_{\mathrm{int}}(y)\in\partial \mc{W}$, which, by the above, can only hold for $x=\Gamma h$ in set of measure zero. Consequently, $\mu(X'\setminus X'')=0$.
	\end{proof}

	\Cref{lemMPhatClosure} implies that when determining $m_{\widehat{\mathcal{P}}}$ for a regular cut-and-project set $\mathcal{P}=\mc{P}(\mc{W},\mc{L})\subset \R^2$, $\mc{W}$ can be replaced with its interior, i.e.\ it may be assumed that $\mc{W}$ is open. In this case, the following lemma holds.
	
	\begin{lem}
		\label{lemRewriteMPhat} Let $\mc{L}=\delta^{1/n}(\Z^ng)$ for some $g\in G$ and let $X=\Gamma\bs \Gamma H_g$. Suppose $\mathcal{P}=\mc{P}(\mc{W},\mc{L})\subset \R^2$ is a regular cut-and-project set with $\mc{W}\subset \overline{\pi_{\mathrm{int}}(\mc{L})}$ open and $\theta(\mc{P})>0$. Then
		\begin{equation}\label{eqnMPHat1}
		m_{\widehat{\mc{P}}}=\inf\{s\geq 0\mid \exists x\in X:\#(\widehat{\mc{P}^x}\cap \mathfrak{C}(\inv{\k_\mc{P}}s))\geq 2\}.\end{equation}
	\end{lem}

	\begin{proof}
		It suffices to show that if $\#(\widehat{\mc{P}^x}\cap \mathfrak{C}(\inv{\k_\mc{P}}s_0))\geq 2$ holds for some $x\in X$, then $\#(\widehat{\mc{P}^{x'}}\cap \mathfrak{C}(\inv{\k_\mc{P}}s_0))\geq 2$ holds for all $x'\in X$ in a set of positive measure. Write $x=\Gamma h$ for some $h\in H_g\subset \ASL{n}{\R}$. Take $n_1,n_2\in\Z^n$ so that $\pi(\delta^{1/n}n_1hg),\pi(\delta^{1/n}n_2hg)$ are linearly independent and belong to $\widehat{\mc{P}^x}\cap\mathfrak{C}(\inv{\kappa_{\mc{P}}}s_0)$. By \cite[Proposition 3.5]{marklof2014free}, we have $\overline{\pi_{\mathrm{int}}(\delta^{1/n}(\Z^nhg))}\subset \overline{\pi_\mathrm{int}(\mc{L})}$ for all $h\in H_g$. For all $x'=\Gamma h'$ with $h'\in H_g$ sufficiently close to $h$ in $\ASL{n}{\R}$ we have that $\pi(\delta^{1/n}n_1h'g),\pi(\delta^{1/n}n_2h'g)$ are linearly independent and belong to $\mc{P}^{x'}\cap\mathfrak{C}(\inv{\kappa_{\mc{P}}}s_0)$ since $\mathfrak{C}(\inv{\kappa_{\mc{P}}}s_0)$ and $\mc{W}$ are open. Since $\mathfrak{C}(\inv{\kappa_{\mc{P}}}s_0)$ is star-shaped with respect to the origin, the claim follows.
	\end{proof}
	
	Given $p_1,p_2\in\R^2$, let $\Delta(p_1,p_2)$ denote the area of the triangle with vertices $0,p_1,p_2$. In view of the $\SL{2}{\R}$-invariance of the process $x\mapsto \mc{P}^x$ we have
	\begin{equation}
	\label{eqnMPHat2}
	m_{\widehat{\mathcal{P}}}=\inf\{s>0\mid \exists x\in X,p_1,p_2\in \mc{P}^x: 0<\Delta(p_1,p_2)<s/\theta(\widehat{\mc{P}})\}
	\end{equation}
	if $\mc{W}$ is open, by \Cref{lemRewriteMPhat} and
	the fact that the area of $\mathfrak{C}(\inv{\k_\mc{P}}s)$ is $s/\theta(\widehat{\mc{P}})$.

	Next we show that the minimal gaps $\widehat{\delta}_T$ at finite horizons converge to the minimal gap under fairly general assumptions.
	
	\begin{lem}
		\label{lemDeltaConvergence}
		Let $\mc{L}=\delta^{1/n}(\Z^ng)$ for some $g\in G$.
		Suppose that $\mathcal{P}=\mc{P}(\mc{W},\mc{L})\subset \R^2$ is a cut-and-project set with $\mc{W}\subset \overline{\pi_{\mathrm{int}}(\mc{L})}$ open, $\theta(\mc{P})>0$ and $m_{\widehat{\mc{P}}}>0$. Then
		\[\lim_{T\to\infty}\widehat{\delta}_T=m_{\widehat{\mc{P}}}.\]
	\end{lem}

	\begin{proof}
		From \eqref{eqnLimMinGap}, it follows that for each $\epsilon>0$ we have $\mu_T([m_{\widehat{\mc{P}}},m_{\widehat{\mc{P}}}+\e))>\delta$ for some $\delta>0$ and all $T$ large enough, i.e.\ the proportion of $\widehat{d}_{T,i}$ that are close to $m_{\widehat{\mc{P}}}$ is positive for all $T$ large enough. Thus, $\limsupl{T}\widehat{\delta}_T\leq m_{\widehat{\mc{P}}}$.
		
		For all $T>0$ large enough we have $\widehat{\xi}_{T,i}-\widehat{\xi}_{T,i-1}\geq (\pi\theta(\widehat{\mc{P}}))^{-1}m_{\widehat{\mc{P}}}T^{-2}$ for all $1\leq i\leq\widehat{N}(T)$ by a modification of \cite[Lemma 15]{marklof2014visibility}; its proof works just as well when it is assumed that $\mc{W}$ is open. Furthermore, by noting that $\mathfrak{C}(\inv{\kappa_{\mc{P}}}s_0)$ is star-shaped with respect to the origin, as in \Cref{lemRewriteMPhat}, it is seen that the assumption $0\notin \mc{P}$ or $0\in \mc{P}^x$ for all $x\in X$ can be omitted from \cite[Lemma 15]{marklof2014visibility}. Thus $\widehat{d}_{T,i}\geq \widehat{N}(T)(\pi\theta(\widehat{\mc{P}}))^{-1}m_{\widehat{\mc{P}}}T^{-2}$ for $T$ large enough, and since the right hand side converges to $m_{\widehat{\mc{P}}}$, it follows that $\liminfl{T}\widehat{\delta}_{T}\geq m_{\widehat{\mc{P}}}$.
	\end{proof}

	The following result shows that for generic translates of a Penrose set, the limiting minimal gap between visible points vanishes.
	
	\begin{prop}
		Let $\e\in \R^2$ be given and consider $\mc{P}_\e'$. For $t\in\R^2$, let $\mc{P}_{\e,t}'=t+\mc{P}_\e'$. Then $m_{\widehat{\mc{P}_{\e,t}'}}=0$ for Lebesgue-almost every $t\in\R^2$.
	\end{prop}
	
	\begin{proof}
		Recall the definition of $\mc{P}_\e'=\mathcal{P}(\mc{W}_\e,\mc{L})$ in \eqref{eqnDefPenroseEps'}, where $\mc{L}=\Z^5 g$. For $t\in\R^2$ we have $\mc{P}_{\e,t}'=\mathcal{P}(\mc{W}_\e,\Z^5g_t)$ with $g_t:=(g,(t,0))\in\ASL{5}{\R}$. By \cite[Prop. 4.5]{marklof2014free}, we have $H_{g_t}=\widetilde{H}_g$ for Lebesgue-almost all $t\in \R^2$. From \cite[Section 2.5]{marklof2014free} we have 
		\[\widetilde{H}_g=g\left\{\left.\left(\begin{pmatrix}
		A_1 & 0 & 0\\
		0 & A_2 & 0\\
		0 & 0 & 1
		\end{pmatrix},(v_1,v_2,0)\right)\right|(A_1,v_1),(A_2,v_2)\in \ASL{2}{\R}\right\}\inv{g}\subset \ASL{5}{\R}.\]
		We now show that for every $t$ with $H_{g_t}=\widetilde{H}_g$ we have $m_{\widehat{\mc{P}_{\e,t}'}}=0$.
		
		Fix $y_1,y_2\in\mc{L}$ such that $\pi(y_1),\pi(y_2)$ are linearly independent and $\pi_{\mathrm{int}}(y_1),\pi_{\mathrm{int}}(y_2)\in \mc{W}_\epsilon$. Let $\sigma_0>0$ be arbitrary and fix $v_1,v_2\in\mathfrak{C}(\inv{\kappa_{\mc{P}_{\e,t}}}\sigma_0)$ which are linearly independent. Since $\ASL{2}{\R}$ acts transitively on pairs of distinct vectors of $\R^2$, there is $(A,v)\in\ASL{2}{\R}$ with $\pi(y_i)A+v=v_i$ for $i\in\{1,2\}$. Let \[h=g\left(\begin{pmatrix}
		A & 0 & 0\\
		0 & I_2 & 0\\
		0 & 0 & 1
		\end{pmatrix},(v-t,0,0)\right)\inv{g}\in \widetilde{H}_g\]
		and $x=\Gamma h$. Then $(v_i,\pi_{\mathrm{int}}(y_i))\in \Z^5hg_t$ and so
		$(\mc{P}_{\e,t}')^x=\mc{P}(\mc{W}_\e,\Z^5hg_t)$ intersects $\mathfrak{C}(\inv{\kappa_{\mc{P}_{\e,t}}}\sigma_0)$ in at least two points which does not lie on the same line through the origin. Hence $\#((\widehat{\mc{P}_{\e,t}')^x}\cap \mathfrak{C}(\inv{\kappa_{\mc{P}_{\e,t}}}\sigma_0))\geq 2$ and since $\mathcal{W}_\e\subset \overline{\pi_{\mathrm{int}}(\mc{L})}$ is open and $\mathfrak{C}(\inv{\kappa_{\mc{P}_{\e,t}}}\sigma_0)$ is open this holds for all $x'=\Gamma h'$ with $h'$ close to $h$ in $\ASL{5}{\R}$. Since $\sigma_0$ was arbitrary, we conclude that $m_{\widehat{\mc{P}_{\e,t}'}}=0$ by \eqref{defnMPhat2}.
	\end{proof}

	In an analogous manner, considering \eqref{eqnACPS}, \eqref{eqnTCPS} and \cite[Section 2.2]{marklof2014free}, it follows that for Lebesgue-almost all translates of an $\mc{A}$-set or $\mc{T}$-set, the limiting minimal gap between visible points is $0$ since in these cases the $H_{g_t}$ is  equal to
	\[\widetilde{H}_g=g\ASL{2}{\R}^2\inv{g}\]
	for a generic translate $t$,
	where $\mc{L}=\delta^{1/4}\Z^4g$ for some $g\in \SL{4}{\R}$.

	\subsection{On $m_{\widehat{\mc{P}}}$ for $\mathcal{A}$-sets} 
	
	Let $\mathcal{W}\subset \R^2$ be a Jordan measurable, open, convex set which contains the origin and consider $\mc{A}_\mc{W}$. Let $\mc{A}=\mc{A}_\mc{W}\inv{A_1}$, with $A_1$ as in \eqref{eqnACPS}. Since $\det(A_1)>0$ we have $m_{\widehat{\mc{A}_\mc{W}}}=m_{\widehat{\mc{A}}}$, as noted in the previous section. Note that $\mc{A}=\mc{P}(\mc{W}\inv{A},\mc{L})$, where $\det(A)=\frac{1}{\sqrt{2}}$ and $\mathcal{L}$ is the Minkowski embedding of $\Z[\sqrt{2}]^2$ in $\R^4$. Pick $\delta>0$ and $g\in \SL{4}{\R}$ so that $\mc{L}=\delta^{1/4}\Z^4g$. 
	
	By \cite[(2.6)]{marklof2014free}, we have $H_g=g\SL{2}{\R}^2\inv{g}$, where $\SL{2}{\R}^2$ is the image of the map 
	\[\iota:\SL{2}{\R}\times \SL{2}{\R}\longrightarrow\SL{4}{\R},\quad (B_1,B_2)\mapsto \begin{pmatrix}B_1 & 0\\ 0 & B_2
	\end{pmatrix}.\] Fix $B_1,B_2\in\SL{2}{\R}$ and let $b=\iota(B_1,B_2)$. If $x=\Gamma g\iota(B_1,B_2)\inv{g}$ then
	\[\mathcal{A}^x=\mathcal{P}(\mc{W}\inv{A},\mc{L}b)=\{xB_1\mid x\in\Z[\sqrt{2}]^2, \sigma(x)B_2\in \mc{W}\inv{A}\}.\]
	
	For convex sets $C_1,C_2\subset \R^2$ containing $0$ let \[T_{C_1,C_2}=\sup_{x\in C_1,y\in C_2}\Delta(x,y).\]
	Let also $T_{C_1}=T_{C_1,C_1}$.
	By \eqref{eqnMPHat2} it follows that
	\begin{align*}
	m_{\widehat{\mc{A}}}&=\inf\left\{s>0\left| \begin{matrix}\exists B_1,B_2\in \SL{2}{\R},x,y\in \Z[\sqrt{2}]^2:  \\  0<\Delta(x,y),\quad xB_1,yB_1\in\mathfrak{C}(\inv{\kappa_\mathcal{P}}s),\quad\sigma(x)B_2,\sigma(y)B_2\in\mc{W}\inv{A}\end{matrix}\right.\right\}\\
	&=\inf\{s>0\mid \exists x,y\in \Z[\sqrt{2}]^2:0<\Delta(x,y)<s/\theta(\widehat{\mathcal{A}}),\quad \Delta(\sigma(x),\sigma(y))<T_{\mathcal{W}\inv{A}}\}\\
	&=\inf\left\{\tfrac{1}{2}\theta(\widehat{\mathcal{A}})|x|:x\in\Z[\sqrt{2}]\smpt{0},\quad|\sigma(x)|<2\sqrt{2}T_\mathcal{W}\right\},
	\end{align*}
	where we make use of the fact that if $x=(x_1,x_2),y=(y_1,y_2)\in\Z[\sqrt{2}]^2$, then $\Delta(x,y)=\frac{1}{2}|x_1y_2-x_2y_1|$ and $\Delta(\sigma(x),\sigma(y))=\frac{1}{2}|\sigma(x_1y_2-x_2y_1)|$.
	Thus, to determine $m_{\widehat{\mc{A}}}$, one must solve a problem of the following type: find the infimum of $|x|$ over non-zero $x\in\Z[\sqrt{2}]$ subject to $|\sigma(x)|<c$ for some fixed $c>0$. This amounts to finding a minimum among finitely many possible values of $|x|$ since $\{(x,\sigma(x))\mid x\in\Z[\sqrt{2}]\}$ is a lattice in $\R^2$. Note that by \Cref{lemSizeOfPrimesA}, this minimum has to be a unit. Indeed, take a non-zero, non-unit $x\in\Z[\sqrt{2}]$ with $|\sigma(x)|<c$ and take a prime $\pi\in\bb{P}$ with $\pi\mid x$. Then, $|x/\pi|<|x|$ and $|\sigma(x/\pi)|\leq\frac{|\sigma(x)|}{\sqrt{2}}<c$, so $x$ cannot be the desired minimum. It follows that the minimum is given by $\lambda^{-m}$, where $\lambda=1+\sqrt{2}$ is the fundamental unit of $\Z[\sqrt{2}]$ and $m$ is the maximal integer such that $\lambda^m<c$. We thus have the following result.
	
	\begin{thm}
		\label{thmmPhatA}
		Let $\mc{W}\subset\R^2$ be an open, convex Jordan measurable set containing the origin. Then
		\[m_{\widehat{\mc{A}_\mc{W}}}=\frac{\lambda^{-m}\theta(\widehat{\mathcal{A}_\mathcal{W}})}{2\sqrt{2}}\]
		where $m$ is the maximal integer such that $\lambda^m<2\sqrt{2}T_\mathcal{W}$.
	\end{thm}

	By observing that if $x=x_1+x_2\z$, $y=y_1+y_2\z$ for $x_1,x_2,y_1,y_2\in\Z[\sqrt{2}]$, then $\Delta(x,y)=\frac{1}{2\sqrt{2}}|x_1y_2-x_2y_1|$ and $\Delta(\sigma(x),\sigma(y))=\frac{1}{2\sqrt{2}}|\sigma(x_1y_2-x_2y_1)|$ it follows that $\frac{\lambda^{-m}}{2\sqrt{2}}$, with $m$ as in \Cref{thmmPhatA}, is equal to  $\min \Delta(\mc{A}_\mc{W})$ where
	\begin{equation}
		\label{eqnDeltaA}
		\Delta(\mc{A}_\mc{W}):=\{\Delta(x,y)\mid x,y\in\Z[\z],x\neq y, \Delta(\sigma(x),\sigma(y))<T_\mc{W}\}.
	\end{equation}
	Furthermore, for each $c>0$ we have $\#(\Delta(\mc{A}_\mc{W})\cap (0,c))<\infty$.
 	
 	\Cref{thmmPhatA} allows us to explicitly calculate $m_{\widehat{\mc{A}_\mc{W}}}$ for a large family of $\mathcal{W}\subset\R^2$; in particular for all $\mathcal{W}$ which are Jordan measurable, open, convex, contain the origin and satisfy $-\mc{W}\subset \sqrt{2}\mathcal{W}$, since then $\theta(\widehat{\mathcal{A}_\mathcal{W}})$ is known explicitly by \Cref{thmDensVisA}.
	
	When $\mathcal{W}$ is the open regular octagon of side length $1$ centered at the origin with sides perpendicularly bisected by the coordinate axes, i.e.\ when $\mathcal{A}_\mathcal{W}$ is the Ammann--Beenker point set, we have $T_\mathcal{W}=\frac{2+\sqrt{2}}{4}$ since the outer radius of $\mathcal{W}$ is $\sqrt{\frac{2+\sqrt{2}}{2}}$. The maximal $m$ with $\lambda^m<2\sqrt{2}T_\mc{W}=1+\sqrt{2}$ is $0$. Recall from \Cref{thmDensVisA} that $\theta(\widehat{\mc{A}_\mc{W}})=\frac{1}{\z_{\Q(\sqrt{2})}(2)}$ and therefore 
	\begin{equation*}
		\label{mPHatAmmann-Beenker}
		m_{\widehat{\mc{A}_\mc{W}}}=\frac{1}{2\sqrt{2}\z_{\Q(\sqrt{2})}(2)}=\frac{24}{\pi^4}=0.2463\ldots,
	\end{equation*}
	by \Cref{thmmPhatA}.
	
	Next, some numerical support for this result is presented. Given $T>0$, recall that $\widehat{N}(T)=\#(B_T(0)\cap\widehat{\mathcal{A}_\mathcal{W}})$. Recall the definition of $\widehat{\xi}_{T,i}$ for $0\leq i\leq \widehat{N}(T)$ in \eqref{eqnXiHat}. Finally recall the normalised gaps given by $\widehat{d}_i=\widehat{N}(T)(\widehat{\xi}_{T,i}-\widehat{\xi}_{T,i-1})$ and $\widehat{\delta}_T=\underset{1\leq i \leq\widehat{N}(T)}{\min}\widehat{d}_{T,i}$.
	
	\begin{figure}[H]
	\centering
	\includegraphics[width=0.6\linewidth]{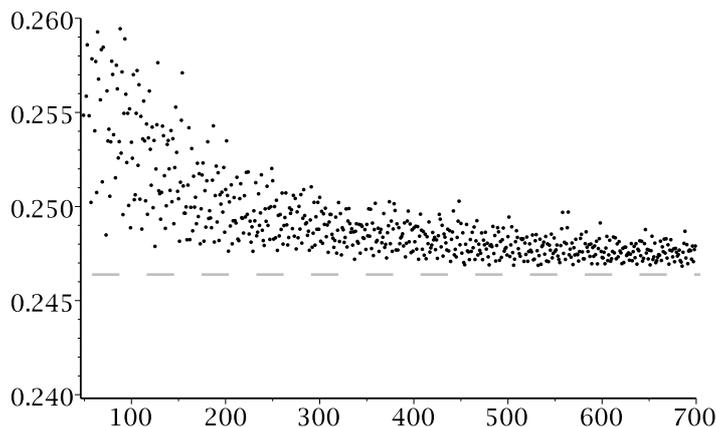}
	\caption{Plot of $(T,\widehat{\delta_T})$ for $T\in\{50,51,\ldots,700\}$, for the Ammann--Beenker point set $\mathcal{A}_\mathcal{W}$. The horizontal dashed line is $y=m_{\widehat{\mc{A}_\mc{W}}}=\frac{24}{\pi^4}$.}
	\label{fig:ABgapGraph700}
	\end{figure}

	\Cref{fig:ABgapGraph700} was produced by generating the $134\,091$ points of $B_{700}(0)\cap\widehat{\mathcal{A}_\mathcal{W}}$ in the closed octant $0\leq y\leq x$ (recall that the Ammann--Beenker point set exhibits eightfold rotational symmetry about the origin) and then calculating $\widehat{\delta}_T$ for $T\in\{50,51,\ldots,700\}$.
	
	\subsection{On $m_{\widehat{\mc{P}}}$ for $\mathcal{P}$-sets}
		
	Let $\z=e^{\frac{2\pi i}{5}}$ and let $\sigma$ be the automorphism of $\Q(\z)$ given by $\z\mapsto \z^2$. Recall the definition of $\mathcal{P}_\epsilon$ in \eqref{eqnDefPenroseEps} and the definition of $\mc{P}_\e'=\mathcal{P}(\mc{W}_\e,\mc{L})$ in \eqref{eqnDefPenroseEps'}. Recall in particular that $\mc{L}=\Z^5g$, where $g\in \mathrm{SO}(5,\R)$ is the matrix whose $(j+1)$-th row is given by $v_j=\sqrt{\frac{2}{5}}(\cos(\frac{2\pi j}{5}),\sin(\frac{2\pi j}{5}),\cos(\frac{4\pi j}{5}),\sin(\frac{4\pi j}{5}),\frac{1}{\sqrt{2}})$. In this section we will give a formula for $m_{\widehat{\mc{P}_\e}}=m_{\widehat{\mc{P}_\e'}}$ for all $\e\in\C$ with $|\e|<0.1$.
	
	By \cite[Section 2.5]{marklof2014free} we have that $H_g=gH\inv{g}$, where
	\[H=\left\{\left.\begin{pmatrix}
	A_1 & 0 & 0\\ 0 & A_2 & 0\\ 0 & 0 & 1
	\end{pmatrix}\right|A_1,A_2\in\SL{2}{\bb{R}}\right\}.\]
	Henceforth we assume that $|\e|<0.1$, so that $\mathcal{W}_{k,\e}$ contains the origin and $\theta(\widehat{\mc{P}_e})$ can be calculated by \Cref{thmDenVisPenrose}. 
	From the structure of $H_g$ and \eqref{eqnMPHat1}, it then follows that
	\begin{align*}m_{\widehat{\mathcal{P}_\e'}}&=\inf\left\{s>0\left| \begin{matrix}\exists k_1,k_2\in\bb{Z}^5, A_1,A_2\in\SL{2}{\R}: \\  \pi(k_1g)A_1,\pi(k_2g)A_1\in\mathfrak{C}(\inv{\kappa_\mathcal{P}}s)\text{ linearly independent},\\\pi_{\mathrm{int}}(k_1g)\begin{pmatrix}
	A_2 & 0\\ 0 & 1
	\end{pmatrix}, \pi_{\mathrm{int}}(k_2g)\begin{pmatrix}
	A_2 & 0\\ 0 & 1
	\end{pmatrix}\in\mathcal{W}_\e\end{matrix}\right.\right\}\\
	&=\min_{1\leq j_1,j_2\leq 4}\inf\left\{s>0\left| \begin{matrix}\exists k_1,k_2\in\bb{Z}^5, \sum_{j=0}^4k_{1,j}=j_1,\sum_{j=0}^4k_{2,j}=j_2: \\  0<\Delta(\pi(k_1g),\pi(k_2g))<\frac{s}{\theta(\widehat{\mc{P}_\e'})}\vspace{1mm}\\\Delta(\pi_{\mathrm{int}}'(k_1g),\pi_{\mathrm{int}}'(k_2g))<\frac{2T_{j_1,j_2}}{5}\end{matrix}\right.\right\},\end{align*}
	where $\pi_{\mathrm{int}}'(x_1,\ldots,x_5)=(x_3,x_4)$. Let $m_{j_1,j_2}$ be the infimum corresponding to $j_1,j_2$ in the last expression, so that $m_{\widehat{\mathcal{P}_\e'}}=\underset{1\leq j_1,j_2\leq 4}{\min}m_{j_1,j_2}$. We can identify $\sqrt{5/2}\pi(kg)$ with $x:=\sum_{j=0}^4k_j\z^j\in\Z[\z]$ and $\sqrt{5/2}\pi_{\mathrm{int}}'(kg)$ with $\sigma(x)$ for every $k\in\Z^5$. Given $x,y\in\Z[\z]$, write $x=x_1+x_2\z$, $y=y_1+y_2\z$ for some $x_1,x_2,y_1,y_2\in\Z[\tau]$. It is then straightforward to show that \[\Delta(x,y)=\tfrac{\sqrt{\t+2}}{4}|x_1y_2-x_2y_1| \quad\text{ and }\quad \Delta(\sigma(x),\sigma(y))=\tfrac{\sqrt{\t+2}}{4\t}|\sigma(x_1y_2-x_2y_1)|.\] It follows that \[m_{j_1,j_2}\geq d_{j_1,j_2}:=\inf\left\{s>0\left|\exists x\in\Z[\tau]:0<|x|<\frac{10s}{\sqrt{\t+2}\theta(\widehat{\mc{P}_\e'})},\quad|\sigma(x)|<\frac{4\t T_{j_1,j_2}}{\sqrt{\t+2}}\right.\right\}.\]
	Note that to determine $d_{j_1,j_2}$ one must solve a problem of the following type: find the infimum of $0<|x|$ over non-zero $x\in\Z[\t]$ subject to $|\sigma(x)|<c$ for some $c>0$. By reasoning as before \Cref{thmmPhatA}, it follows that the infimum is a minimum and must be a unit.
	
	From the definition of $d_{j_1,j_2}$ it is seen that $d_{j_1,j_2}$ is minimal when $T_{j_1,j_2}$ is maximal, in which case $j_1,j_2\in\{2,3\}$. Fix such $j_1,j_2$. Suppose $x=\tau^m$ with $m\in\Z$ gives the minimal $|x|$ subject to $|\sigma(x)|<\frac{4\t T_{j_1,j_2}}{\sqrt{\t+2}}$. If $j_1=2$ let $x=\zeta^2+\zeta^3=-\t$ and $y=y_1-\t^{m-1}\z$ where $y_1\in\Z[\t]$ is chosen so that $y=\sum_{j=0}^4k_j\z^j$ with $\sum_{j=0}^4k_j=j_1$. If $j_2=3$ let $x=\t$ and $y=y_1+\t^{m-1}\z$ where $y_1\in\Z[\t]$ is chosen so that $y=\sum_{j=0}^4k_j\z^j$ with $\sum_{j=0}^4k_j=j_2$. This shows that $m_{j_1,j_2}=d_{j_1,j_2}$ and hence we have the following result.
		
	\begin{thm}
		\label{thmMPHatPenrose}
		For $\e\in\C$ with $|\e|<0.1$ we have
		\[m_{\widehat{\mc{P}_\e}}=\frac{\t^{-m}\sqrt{\t+2}\,\theta(\widehat{\mc{P}_\e})}{4},\]
		where $m$ is the maximal integer such that $\t^m<\frac{4\t}{\sqrt{\t+2}}\underset{{2\leq j_1,j_2\leq 3}}{\max}T_{j_1,j_2}$.
	\end{thm}

	Let 
	\begin{equation}
	\label{eqnDeltaP}
	\Delta(\mc{P}_\e)=\left\{\Delta(x,y)\mid x,y\in\Z[\z],x\neq y, \Delta(\sigma(x),\sigma(y))<\underset{{2\leq j_1,j_2\leq 3}}{\max}T_{j_1,j_2}\right\}.
	\end{equation}
	If $\e\in \C$ satisfies $|\e|<0.1$ it follows from \Cref{thmMPHatPenrose} that $m_{\widehat{\mc{P}_\e}}=\theta(\widehat{\mc{P}_\e})\min \Delta(\mc{P}_\e)$. Furthermore, $\#(\Delta(\mc{P}_\e)\cap(0,c))<\infty$ for each $c>0$.
	
	We illustrate \Cref{thmMPHatPenrose} by an example. As in the discussion after \Cref{thmDenVisPenrose}, let $\gamma=\frac{1}{101}(2,1,-2-2,1)$ and set $\e_0=\sum_{j=0}^4\gamma_j\z^{2j}$. Recall that $\mc{P}_{\e_0}$ is the vertex set of a rhombic Penrose tiling and that $|\e_0|<0.1$. It is easily verified numerically that $T_{j_1,j_2}$ is maximal when $(j_1,j_2)\in\{(2,3),(3,2)\}$ and that then $T_{j_1,j_2}=1.2554\ldots$ 
	and hence $\frac{4\t T_{j_1,j_2}}{\sqrt{\t+2}}=4.2718\ldots$. Recall from the discussion after \Cref{thmDenVisPenrose} that $\theta(\widehat{\mc{P}_{\e_0}})=0.6843\ldots$. Since $\t^3<4.2718\ldots<\t^4$, \Cref{thmMPHatPenrose} gives
	\[m_{\widehat{\mc{P}_{\e_0}}}=\frac{\t^{-3}\sqrt{\t+2}\,\theta(\widehat{\mc{P}_{\e_0}})}{4}=0.07681\ldots\]
	
	We now present some numerical support for the above value of the minimal gap. \Cref{fig:gapsInPenrose} was produced by generating the $8\,599\,221$ visible points of $\widehat{\mc{P}_{\e_0}}$ in $B_{2000}(0)$ and then calculating $\widehat{\delta_T}$ numerically for all $T\in\{100+10i\mid i\in\{0,1,\ldots,190\}\}$. 
	
	\begin{figure}[H]
		\centering
		\includegraphics[width=0.6\linewidth]{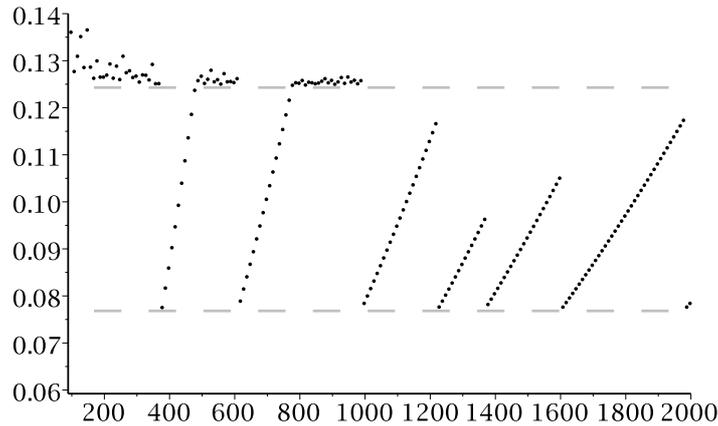}
		\caption{Plot of $(T,\widehat{\delta_T})$ for $T\in\{100+10i\mid i\in\{0,1,\ldots,190\}\}$, for $\mc{P}_{\e_0}$ with $\e_0=\sum_{j=0}^4\gamma_j\z^{2j}$, $\gamma=\frac{1}{101}(2,1,-2-2,1)$. The horizontal dashed lines are $y=m_{\widehat{\mc{P}_{\e_0}}}$ and $y=\t m_{\widehat{\mc{P}_{\e_0}}}$.}
		\label{fig:gapsInPenrose}
	\end{figure}
	
 	We conclude this paper with a comparison of \Cref{fig:ABgapGraph700} and \Cref{fig:gapsInPenrose}. Let $\mc{A}_\mc{W}$ be the Ammann--Beenker point set. By \Cref{thmmPhatA}, we have $m_{\widehat{\mc{A}_\mc{W}}}=\Delta_1\theta(\widehat{\mc{A}_\mc{W}})$, where $\Delta_1:=\min \Delta(\mc{A}_\mc{W})=\frac{1}{2\sqrt{2}}$ (cf.\ \eqref{eqnDeltaA}). Consider a gap $\widehat{d}_{T,i}$ formed by $x,x'\in \widehat{\mc{A}_\mc{W}}\subset \Z[\z]$, $\norm{x}\leq \norm{x'}\leq T$. Let $v$ be the angle between $x,x'$ so that $\widehat{d}_{T,i}=\widehat{N}_T\frac{v}{2\pi}$. By taking $T$ large, we may suppose that $v$ is small, so that $v$ is close to $\sin v$. From $\Delta(x,x')=\frac{1}{2}\norm{x}\norm{x'}\sin v\in\Delta(\mc{A}_\mc{W})$, $\widehat{N}_T\sim \theta(\widehat{\mc{A}_\mc{W}})T^2\pi$ and the fact that $\Delta(\mc{A}_\mc{W})\cap (0,c)$ is finite for each $c>0$, it follows that for large $T$, if $\widehat{d}_{T,i}$ is close to $m_{\widehat{\mc{A}_\mc{W}}}$, then $\norm{x}$, $\norm{x'}$ must both be close to $T$ and $\Delta(x,x')=\Delta_1$. That is, the triangle formed by $0,x,x'$ must have area $\Delta_1$ and be nearly isosceles. In this case $\Delta_1^\sigma:=\Delta(\sigma(x),\sigma(x'))=\frac{1}{2\sqrt{2}}$.
 	
 	Similarly, if $T$ is large, and $y,y'\in \widehat{\mc{P}_{\e_0}}$ forms a gap $\widehat{d}_{T,i}$ which is close to $m_{\widehat{\mc{P}_{\e_0}}}$, then the triangle formed by $0,y,y'$ must be nearly isosceles and have area equal to $\Delta_2:=\min \Delta(\mc{P}_{\e_0})=\frac{\t^{-3}\sqrt{\t+2}}{4}$ (cf.\ \eqref{eqnDeltaP}). We then have $\Delta_2^\sigma:=\Delta(\sigma(y),\sigma(y'))=\frac{\t^{2}\sqrt{\t+2}}{4}$.
 	
 	Assume now that for those $T$ we have considered numerically, the points $\sigma(x)$, $x\in \widehat{\mc{A}_{\mc{W}}}\cap B_T(0)$ are well-distributed in $\mc{W}$, and that $\sigma(y)$, $y\in\widehat{\mc{P}_{\e_0}}\cap B_T(0)$ are well-distributed in $\mc{W}_{1,\e_0},\ldots,\mc{W}_{4,\e_0}$. Observe that the difference  $T_{2,3}-\Delta_2^\sigma=0.01\ldots$ is quite small. Thus, for points $y,y'\in \widehat{\mc{P}_{\e_0}}\cap B_T(0)$ with $\Delta(y,y')=\Delta_2$, the points $\sigma(y),\sigma(y')$ are forced to be near vertices of $\mc{W}_{2,\e_0}$ and $\mc{W}_{3,\e_0}$, respectively. On the other hand, the difference $T_\mc{W}-\Delta_1^\sigma=\frac{1}{2}$ is substantially larger, so the restriction of the location in $\mc{W}$ of the conjugates of points $x,x'\in \widehat{\mc{A}_\mc{W}}$ with $\Delta(x,x')=\Delta_1$ is not as severe. Under the above well-distribution assumption, we find a possible explanation to the observation that it seems more likely that $\widehat{\delta}_T$ is close to $m_{\widehat{\mc{P}}}$ in the case of $\mc{A}_\mc{W}$ (see \Cref{fig:ABgapGraph700}) than in the case of $\mc{P}_{\e_0}$ (see \Cref{fig:gapsInPenrose}) for comparable values of $T$. It should be recalled that, in both cases, $\widehat{\delta}_T\to m_{\widehat{\mc{P}}}$ by \Cref{lemDeltaConvergence}.
	
	\bibliographystyle{siam}
	\bibliography{bibl}	

\end{document}